\documentclass[twocolumn]{article}
\usepackage{}
\usepackage{soul}
\usepackage{amsfonts}
\usepackage[utf8]{inputenc}
\usepackage[T1]{fontenc}
\usepackage{amsmath, amssymb, amsthm, amsfonts, cases}
\usepackage{mathrsfs}
\usepackage{amsmath}
\usepackage{url}
\usepackage{authblk}

\usepackage[usenames]{color}
\usepackage{geometry}
\usepackage{graphicx}
\usepackage{manfnt}
\usepackage{epsfig}
\usepackage{lipsum}
\usepackage{indentfirst}
\usepackage{appendix}
\usepackage[numbers]{natbib}
\usepackage[
colorlinks=true,
linkcolor=black,
citecolor=black
]{hyperref}

\theoremstyle{plain}
\newtheorem{thm}{Theorem}[section]

\newtheorem{pro}[thm]{Problem}
\newtheorem{lem}[thm]{Lemma}

\theoremstyle{definition}
\newtheorem{defn}[thm]{Definition}
\newtheorem{ass}{Assumption}[section]
\newtheorem{rmk}[thm]{Remark}

\newcommand{\midsmall}{\fontsize{11pt}{13pt}\selectfont}

\renewcommand{\theequation}{\thesection.\arabic{equation}}
\makeatletter\@addtoreset{equation}{section} \makeatother

\begin{document}

	\title{Discrete-Time LQ Stochastic Two-Person Non-zero-Sum Difference Games with Random Coefficients:~Open-Loop Nash Equilibrium
		\thanks{Qingxin Meng was supported by the Key Projects of Natural Science Foundation of Zhejiang Province (No. LZ22A010005) and the National
			Natural Science Foundation of China ( No.12271158).   Xun Li is supported by RGC of Hong Kong grants 15221621, 15226922 and 15225124, and partially from PolyU 1-ZVXA.}}
	
	\date{}
	
	\author[a]{Yiwei Wu}
	
	\affil[a]{\small{School of Mathematical Sciences, South China Normal University, Guangzhou 510631, China}}
	\author[b]{Xun Li}
	\author[c]{Qingxin Meng\footnote{Corresponding author.
			\authorcr
			\indent E-mail address: elorywoo@gmail.com (Y.Wu), li.xun@polyu.edu.hk (X. Li), mqx@zjhu.edu.cn (Q.Meng)}}
	
	\affil[b]{\small{Department of Applied Mathematics, The Hong Kong Polytechnic University, Hong Kong, PR  China}}
	\affil[c]{\small{Department of Mathematics, Huzhou University, Zhejiang 313000,PR  China}}
\maketitle

\begin{abstract}
This paper presents a pioneering investigation into discrete-time two-person non-zero-sum linear quadratic (LQ) stochastic games with random coefficients. We derive necessary and sufficient conditions for the existence of open-loop Nash equilibria using convex variational calculus. To obtain explicit expressions for the Nash equilibria, we introduce fully coupled forward-backward stochastic difference equations (FBS$\Delta$E, for short), which provide a dual characterization of these Nash equilibria. Additionally, we develop non-symmetric stochastic Riccati equations that decouple the stochastic Hamiltonian system for each player, enabling the derivation of closed-loop feedback forms for open-loop Nash equilibrium strategies. A notable aspect of this research is the complete randomness of the coefficients, which results in the corresponding Riccati equations becoming fully nonlinear higher-order backward stochastic difference equations. It distinguishes our non-zero-sum difference game from the deterministic case, where the Riccati equations reduce to algebraic forms.
\end{abstract}
	

\section{Introduction}

Since the 1950s, game theory has found wide applications in economics, control theory, and engineering. Among its many models, LQ games stand out for their analytical tractability and wide applicability. A key problem in game theory is the characterization of Nash equilibria, where each player seeks to optimize their cost functional under fixed strategies of others. In two-player zero-sum games, one player's gain is the other's loss, while non-zero-sum games introduce more complex dynamics, as players' outcomes are not strictly opposed. Such settings reflect realistic scenarios such as resource allocation and industrial competition, where cooperation and conflict may coexist. Non-zero-sum games also often yield multiple, possibly non-unique Nash equilibria, adding to the complexity.

LQ game problems are typically studied under continuous-time or discrete-time frameworks. In the continuous-time case, Starr and Ho \cite{1} first introduced non-zero-sum differential games. Subsequently, numerous works have addressed LQ games with deterministic dynamics \cite{2,3,22,23}.~Huang et al. \cite{8} developed novel methods for mean-field games with partially observable forward-backward stochastic systems, advancing techniques to handle common noise in such frameworks.~Mou and Yong \cite{4} used Hilbert space methods to solve two-player zero-sum stochastic differential games, while Sun and Yong \cite{6} derived Riccati-based solutions for non-zero-sum settings, focusing on open-loop and closed-loop Nash equilibria. Hamadene \cite{16,17} further established connections between stochastic LQ games and forward-backward stochastic differential equations (FBSDE, for short), while Hamadene and Mu \cite{18} addressed the existence of Nash equilibria with unbounded coefficients. For a broader overview, we refer to \cite{1,2,3,7,9,10,12,13,14,HH,24,25,26}.

Discrete-time LQ games are equally important due to their practical relevance.~Zhang and Liu \cite{27} proposed a dynamic compensation method for continuous descriptor systems, providing theoretical foundations for handling parameter uncertainties.~Gao and Lin \cite{9} addressed two-player discrete-time non-zero-sum LQ games with deterministic coefficients, establishing the existence of open-loop and closed-loop equilibria. However, the structure of their Riccati equations is fundamentally different from ours, as their setting excludes the randomness in coefficients.

The treatment approaches of stochastic maximum principles (SMP) \cite{21,32,33} for fully coupled forward-backward stochastic systems with correlated noise have a wide range of applications. These developments provide crucial theoretical tools for handling control variables that simultaneously affect both diffusion processes and correlated noise through conditional expectation operators - methodological insights we extend to discrete-time stochastic games with random coefficients.~This connection is particularly relevant as we derive our stochastic Hamiltonian system and the associated adjoint equations.

Despite extensive work on zero-sum and deterministic-coefficient games, discrete-time non-zero-sum LQ games with \textit{random coefficients} remain underexplored. In practice, many control systems involve random parameters due to external noise or model uncertainty, making it crucial to consider random coefficients.~In this setting, the Riccati equations become more complex. Different from the deterministic case that only involves conditional expectations, the stochastic Riccati equations here also include higher-order terms which introduce new theoretical challenges in solving for the Nash equilibrium.

\textbf{This paper addresses these challenges by focusing on discrete-time non-zero-sum LQ stochastic games with random coefficients.} We study the FBS$\Delta$E, known as stochastic Hamiltonian systems, that characterize Nash equilibria. We derive necessary and sufficient conditions for the existence of open-loop equilibria and construct non-symmetric stochastic Riccati equations that yield closed-loop feedback representations of the strategies.

The main contributions of this paper are as follows:
\begin{itemize}
	\item We establish a SMP for discrete-time non-zero-sum games with random coefficients. The Hamiltonian systems for each player are formulated, and the duality representation of the Nash equilibrium is presented.
	\item We propose a set of stochastic non-symmetric Riccati equations to decouple the forward-backward system. Solving these Riccati equations allows us to represent the open-loop equilibrium in a state feedback (closed-loop) form.
	\item Our work reveals essential differences between Riccati equations with random coefficients and those in the deterministic case, thus contributing new insights to the theory of stochastic games.
\end{itemize}

Due to space constraints, this paper focuses on open-loop Nash equilibria. Future research will address the closed-loop setting in detail.

The remainder of the paper is organized as follows: Section~2 formulates the problem and introduces necessary assumptions. Section~3 derives the variational inequality characterizing the Nash equilibrium. Section~4 presents the stochastic Hamiltonian system and its decoupling via stochastic Riccati equations. Section~5 concludes the paper.

\section{Problem Statement}\label{sec:problem_statement}
	
	
	Let $\Omega$ be a fixed sample space and $N$ a positive integer. Define the index sets $\mathcal{T} := \{0, 1, \ldots, N-1\}$ and $\overline{\mathcal{T}} := \{0, 1, \ldots, N\}$. Consider a complete filtered probability space $(\Omega, \mathcal{F}, \mathbb{F}, \mathbb{P})$, where the filtration $\mathbb{F} := \{\mathcal{F}_k\}_{k \in \mathcal{T}}$ is $\mathbb{P}$-complete. Let $\{\omega_k\}_{k \in \mathcal{T}}$ be a given martingale difference sequence adapted to $\mathbb{F}$. For each $k \in \mathcal{T}$, we assume the following conditions hold:
\begin{equation}\label{eq:martingale_diff}
\mathbb{E}[\omega_k \mid \mathcal{F}_{k-1}] = 0, \quad
\mathbb{E}[\omega_k^2 \mid \mathcal{F}_{k-1}] = 1, \quad
\mathbb{E}[\omega_k^4] < \infty.
\end{equation}
Throughout this paper, we adopt the specific choice $\mathcal{F}_k := \sigma(\omega_0, \ldots, \omega_k)$ for $k \in \mathcal{T}$, and define $\mathcal{F}_{-1} := \{\emptyset, \Omega\}$ for completeness.

We begin by introducing the notation and functional spaces used throughout this paper. Let $\mathbb{R}^n$ denote the $n$-dimensional Euclidean space, and $\mathbb{R}^{n \times m}$ the set of real $n \times m$ matrices. Denote by $\mathcal{S}^n \subset \mathbb{R}^{n \times n}$ the space of symmetric $n \times n$ matrices, and by $\mathcal{S}_+^n \subset \mathcal{S}^n$ the cone of symmetric non-negative definite matrices. For a matrix $U$, we write $U^{\top}$ for its transpose, $U^{-1}$ for its inverse (when it exists). Throughout this paper, $|\cdot|$ denotes the standard Euclidean norm for vectors in $\mathbb{R}^n$ and the Frobenius norm for matrices in $\mathbb{R}^{n \times m}$. Let $E$ denote a finite-dimensional real vector space (e.g., $\mathbb{R}^n$ or $\mathbb{R}^{n \times m}$). The following function spaces will be used:

\begin{itemize}
    \item \textbf{$L^2_{\mathcal{F}_{N-1}}(\Omega; E)$}: the space of all $E$-valued, $\mathcal{F}_{N-1}$-measurable random variables $\xi$ such that
    \[
    \|\xi\|_{L^2_{\mathcal{F}_{N-1}}(\Omega; E)} := \left[\mathbb{E} |\xi|^2 \right]^{1/2} < \infty.
    \]

    \item \textbf{$L^\infty_{\mathcal{F}_{N-1}}(\Omega; E)$}: the space of all essentially bounded $E$-valued, $\mathcal{F}_{N-1}$-measurable random variables.

    \item \textbf{$L^2_{\mathbb{F}}(\mathcal{T}; E)$}: the space of all $E$-valued stochastic processes $f(\cdot) = \{f_k\}_{k \in \mathcal{T}}$, where each $f_k$ is $\mathcal{F}_{k-1}$-measurable, such that
    \[
    \|f(\cdot)\|_{L^2_{\mathbb{F}}(\mathcal{T}; E)} := \left[\mathbb{E} \left(\sum_{k=0}^{N-1} |f_k|^2 \right) \right]^{1/2} < \infty.
    \]

    \item \textbf{$L^\infty_{\mathbb{F}}(\mathcal{T}; E)$}: the space of all essentially bounded $E$-valued stochastic processes $f(\cdot) = \{f_k\}_{k \in \mathcal{T}}$, where each $f_k$ is $\mathcal{F}_{k-1}$-measurable.
\end{itemize}

We begin by introducing the notation and functional spaces used throughout this paper. Let $\mathbb{R}^n$ denote the $n$-dimensional Euclidean space, and $\mathbb{R}^{n \times m}$ the set of real $n \times m$ matrices. Denote by $\mathcal{S}^n \subset \mathbb{R}^{n \times n}$ the space of symmetric $n \times n$ matrices, and by $\mathcal{S}_+^n \subset \mathcal{S}^n$ the cone of symmetric non-negative definite matrices. For a matrix $U$, we write $U^{\top}$ for its transpose, $U^{-1}$ for its inverse (when it exists). Throughout this paper, $|\cdot|$ denotes the standard Euclidean norm for vectors in $\mathbb{R}^n$ and the Frobenius norm for matrices in $\mathbb{R}^{n \times m}$. Let $E$ denote a finite-dimensional real vector space (e.g., $\mathbb{R}^n$ or $\mathbb{R}^{n \times m}$). The following function spaces will be used:

\begin{itemize}
    \item \textbf{$L^2_{\mathcal{F}_{N-1}}(\Omega; E)$}: the space of all $E$-valued, $\mathcal{F}_{N-1}$-measurable random variables $\xi$ such that
    $$
    \|\xi\|_{L^2_{\mathcal{F}_{N-1}}(\Omega; E)} := \left[\mathbb{E} |\xi|^2 \right]^{1/2} < \infty.
    $$

    \item \textbf{$L^\infty_{\mathcal{F}_{N-1}}(\Omega; E)$}: the space of all essentially bounded $E$-valued, $\mathcal{F}_{N-1}$-measurable random variables.

    \item \textbf{$L^2_{\mathbb{F}}(\mathcal{T}; E)$}: the space of all $E$-valued stochastic processes $f(\cdot) = \{f_k\}_{k \in \mathcal{T}}$, where each $f_k$ is $\mathcal{F}_{k-1}$-measurable, such that
    $$
    \|f(\cdot)\|_{L^2_{\mathbb{F}}(\mathcal{T}; E)} := \left[\mathbb{E} \left(\sum_{k=0}^{N-1} |f_k|^2 \right) \right]^{1/2} < \infty.
    $$

    \item \textbf{$L^\infty_{\mathbb{F}}(\mathcal{T}; E)$}: the space of all essentially bounded $E$-valued stochastic processes $f(\cdot) = \{f_k\}_{k \in \mathcal{T}}$, where each $f_k$ is $\mathcal{F}_{k-1}$-measurable.
\end{itemize}

We consider the following linear controlled stochastic difference equation (S$\Delta$E):
\begin{equation}\label{eq:2.1}
\left\{
\begin{aligned}
x_{k+1} &= A_k x_k + B_k u_k + C_k v_k + b_k \\
         & \quad + \left(D_k x_k + E_k u_k + F_k v_k + \sigma_k\right) \omega_k, \\
x_0 &= \xi \in \mathbb{R}^n, \quad k \in \mathcal{T},
\end{aligned}
\right.
\end{equation}
where $x_k \in \mathbb{R}^n$ is the state process, $u_k \in \mathbb{R}^m$ is the control of Player 1, and $v_k \in \mathbb{R}^l$ is the control of Player 2. The matrices $A_k, B_k, C_k, D_k, E_k, F_k$ are $\mathcal{F}_{k-1}$-measurable random matrices of appropriate dimensions, and $b_k, \sigma_k \in \mathbb{R}^n$ are $\mathcal{F}_{k-1}$-measurable random vectors.

We define the admissible control spaces $\mathcal{U}$ and $\mathcal{V}$ as follows:
\begin{equation}\label{eq:2.2}
\begin{split}
\mathcal{U} &:= L^2_{\mathbb{F}}(\mathcal{T}; \mathbb{R}^m) \\
&= \left\{ \mathbf{u} = \{u_k\}_{k \in \mathcal{T}} ~\bigg|~ u_k \text{ is } \mathcal{F}_{k-1}\text{-measurable},~ \mathbb{E} \sum_{k=0}^{N-1} |u_k|^2 < \infty \right\},
\end{split}
\end{equation}
\begin{equation}\label{eq:2.3}
\begin{split}
\mathcal{V} &:= L^2_{\mathbb{F}}(\mathcal{T}; \mathbb{R}^l) \\
&= \left\{ \mathbf{v} = \{v_k\}_{k \in \mathcal{T}} ~\bigg|~ v_k \text{ is } \mathcal{F}_{k-1}\text{-measurable},~ \mathbb{E} \sum_{k=0}^{N-1} |v_k|^2 < \infty \right\}.
\end{split}
\end{equation}

Here, $\mathbf{u} \in \mathcal{U}$ and $\mathbf{v} \in \mathcal{V}$ are called admissible control processes of Player 1 and Player 2, respectively. The process $x = \{x_k\}_{k \in \overline{\mathcal{T}}}$ is called the admissible state process corresponding to the control pair $(\mathbf{u}, \mathbf{v})$. The triple $(x, \mathbf{u}, \mathbf{v})$ is referred to as an admissible 3-tuple.

Our study focuses on non-zero-sum stochastic difference games, where the players' objectives do not directly oppose each other. In such frameworks, both players may pursue strategies that can simultaneously yield favorable outcomes for themselves, thereby fostering cooperation or mutual benefit. It introduces complexity to the optimal control strategies, as each player's decisions can influence the overall system dynamics in ways that may benefit both parties. The quadratic cost functionals for the initial state $x_0 = \xi$ and the admissible control processes $\mathbf{u}=\left(u_{0}, u_{1}, \ldots, u_{N-1}\right)$ and $\mathbf{v}=\left(v_{0}, v_{1}, \ldots, v_{N-1}\right)$ are defined as follows:

For player 1, the cost functional is given by:
\begin{equation}\label{eq:2.4}
\begin{split}
\mathcal{J}_1(0, &\xi; \mathbf{u}, \mathbf{v})
\\=& \frac{1}{2} \mathbb{E}\bigg\{\langle G_{N} x_{N}, x_{N}\rangle + 2\langle g_{N}, x_{N}\rangle + \sum_{k=0}^{N-1} \bigg[\langle Q_{k}x_{k}, x_{k}\rangle \\
&+ 2\langle L_{k}^{\top} u_{k}, x_{k}\rangle + \langle R_{k}u_{k}, u_{k}\rangle + 2\langle q_{k}, x_{k}\rangle + 2\langle \rho_{k}, u_{k}\rangle\bigg]\bigg\}.
\end{split}
\end{equation}

For player 2, the cost functional is similarly defined as:
\begin{equation}\label{eq:2.5}
\begin{split}
\mathcal{J}_2(0, &\xi; \mathbf{u}, \mathbf{v})
\\=& \frac{1}{2} \mathbb{E}\bigg\{\langle H_{N} x_{N}, x_{N}\rangle + 2\langle h_{N}, x_{N}\rangle + \sum_{k=0}^{N-1} \bigg[\langle P_{k}x_{k}, x_{k}\rangle \\
&+ 2\langle M_{k}^{\top} v_{k}, x_{k}\rangle + \langle S_{k}v_{k}, v_{k}\rangle + 2\langle p_{k}, x_{k}\rangle + 2\langle \theta_{k}, v_{k}\rangle\bigg]\bigg\}.
\end{split}
\end{equation}

In these cost functionals \eqref{eq:2.4} and \eqref{eq:2.5}, the random weighting coefficients $Q_{k}, R_{k}, P_{k}, S_{k}$ satisfy $Q_{k}^{\top} = Q_{k}$, $R_{k}^{\top} = R_{k}$, $P_{k}^{\top} = P_{k}$, $S_{k}^{\top} = S_{k}$; $G_{N}$ and $H_{N}$ are $\mathcal{F}_{N-1}$-measurable random symmetric real matrices; $g_{N}$ and $h_{N}$ are $\mathcal{F}_{N-1}$-measurable random variables; and $q_{k}, \rho_{k}, p_{k}, \theta_{k}$ are vector-valued $\mathcal{F}_{k-1}$-measurable processes.

\begin{rmk}\label{rmk:2.1}
Many previous papers have examined discrete-time stochastic systems with deterministic coefficients. Our most significant innovation in this article is related to the study of discrete-time systems with random coefficients. We identify the strategies for both players in discrete-time non-zero-sum stochastic difference games with random coefficients and nonhomogeneous terms.
\end{rmk}

Next, we impose the following standard assumptions on the coefficients.
\begin{ass}\label{ass:2.1}
The coefficient processes satisfy
\begin{equation}\label{eq:2.60}
\left\{
\begin{aligned}
&A(\cdot), D(\cdot) \in L^\infty_{\mathbb{F}}(\mathcal{T}; \mathbb{R}^{n \times n}), \\
&B(\cdot), E(\cdot) \in L^\infty_{\mathbb{F}}(\mathcal{T}; \mathbb{R}^{n \times m}), \\
&C(\cdot), F(\cdot) \in L^\infty_{\mathbb{F}}(\mathcal{T}; \mathbb{R}^{n \times l}), \\
&b(\cdot), \sigma(\cdot) \in L^2_{\mathbb{F}}(\mathcal{T}; \mathbb{R}^n).
\end{aligned}
\right.
\end{equation}
\end{ass}

\begin{ass}\label{ass:2.2}
The weighting coefficient processes satisfy
\begin{equation}\label{eq:2.61}
\left\{
\begin{aligned}
&Q(\cdot), P(\cdot) \in L^\infty_{\mathbb{F}}(\mathcal{T}; \mathbb{R}^{n \times n}), \quad L(\cdot) \in L^\infty_{\mathbb{F}}(\mathcal{T}; \mathbb{R}^{m \times n}), \\
&M(\cdot) \in L^\infty_{\mathbb{F}}(\mathcal{T}; \mathbb{R}^{l \times n}), \qquad\quad R(\cdot) \in L^\infty_{\mathbb{F}}(\mathcal{T}; \mathbb{R}^{m \times m}), \\
&S(\cdot) \in L^\infty_{\mathbb{F}}(\mathcal{T}; \mathbb{R}^{l \times l}), \qquad\quad~~ G_N, H_N \in L^\infty_{\mathcal{F}_{N-1}}(\Omega; \mathbb{R}^{n \times n}), \\
&g_N, h_N \in L^2_{\mathcal{F}_{N-1}}(\Omega; \mathbb{R}^n), \qquad q(\cdot), p(\cdot) \in L^2_{\mathbb{F}}(\mathcal{T}; \mathbb{R}^n), \\
&\rho(\cdot) \in L^2_{\mathbb{F}}(\mathcal{T}; \mathbb{R}^m), \qquad\qquad\quad\theta(\cdot) \in L^2_{\mathbb{F}}(\mathcal{T}; \mathbb{R}^l).
\end{aligned}
\right.
\end{equation}
\end{ass}

Given Assumption \ref{ass:2.1}, for any initial state $\xi \in \mathbb{R}^{n}$, $\mathbf{u} \in \mathcal{U}$ and $\mathbf{v} \in \mathcal{V}$, the state equation \eqref{eq:2.1} has a unique solution, denoted as $ x(\cdot) \equiv \{ x^{(\xi, \mathbf{u}, \mathbf{v})}_k \}_{k=0}^N \in L_{\mathbb{F}}^2(\overline{\mathcal{T}};\mathbb{R}^{n})$. Furthermore, under Assumptions \ref{ass:2.1} and \ref{ass:2.2}, for any initial state $\xi \in \mathbb{R}^{n}$, $\mathbf{u} \in \mathcal{U}$ and $\mathbf{v} \in \mathcal{V}$, it is easy to check that the cost functionals \eqref{eq:2.4} and \eqref{eq:2.5} are well-defined. Therefore, the discrete-time non-zero-sum stochastic difference games with random coefficients are well-defined.

\begin{ass}\label{ass:2.3}
    For all $ k \in \mathcal{T} $, there exists a constant $\delta > 0$ such that
    \begin{equation}\label{eq:2.6}
        \left\{
        \begin{aligned}
            &G_{N},~H_{N} \succeq 0,~~ R_{k},~S_{k} \succeq \delta I, \\
            &Q_{k} - L^{\top}_{k} R^{-1}_{k} L_{k} \succeq 0,~~ \\
            &P_{k} - M^{\top}_{k} S^{-1}_{k} M_{k} \succeq 0, \quad k \in \mathcal{T}.
        \end{aligned}
        \right.
    \end{equation}
\end{ass}

Now, we formulate our non-zero-sum stochastic difference games as follows:
\begin{pro}[\textbf{DTSDG}]\label{pro:2.2}
    For any given initial value $\xi \in \mathbb{R}^n$, find a pair of admissible control $(\mathbf{u}^*, \mathbf{v}^*)\in (\mathcal{U}\times \mathcal{V})$ such that
    \begin{equation}\label{eq:2.7}
        \left\{
        \begin{aligned}
            V_1(0, \xi) &\triangleq \mathcal{J}_1(0, \xi, \mathbf{u}^*, \mathbf{v}^*) = \inf_{\mathbf{u} \in \mathcal{U}} \mathcal{J}_1(0, \xi, \mathbf{u}, \mathbf{v}^*), \\
            V_2(0, \xi) &\triangleq \mathcal{J}_2(0, \xi, \mathbf{u}^*, \mathbf{v}^*) = \inf_{\mathbf{v} \in \mathcal{V}} \mathcal{J}_2(0, \xi, \mathbf{u}^*, \mathbf{v}).
        \end{aligned}
        \right.
    \end{equation}
\end{pro}

Any pair $(\mathbf{u}^*, \mathbf{v}^*) \in \mathcal{U} \times \mathcal{V}$ that satisfies the conditions in \eqref{eq:2.7} is termed an open-loop Nash equilibrium point of Problem \ref{pro:2.2} (DTSDG), with the corresponding state trajectory denoted as $ x^*(\cdot) \equiv \{ x^{(\xi, \mathbf{u}^*, \mathbf{v}^*)}_k \}_{k=0}^N $. The value functions for player 1 and player 2 are defined as $ V_1(0, \xi) $ and $ V_2(0, \xi) $, respectively.

In the special case where $ b_k, \sigma_k, g_N, h_N, q_k, p_k, \rho_k, \theta_k = 0 $ and $ \xi = 0 $, the corresponding cost functionals and value functions are denoted by $ \mathcal{J}_1^{0}(0, 0, \mathbf{u}, \mathbf{v}) $ and $ V_1^{0}(0, 0) $ for player 1, and $ \mathcal{J}_2^{0}(0, 0, \mathbf{u}, \mathbf{v}) $ and $ V_2^{0}(0, 0) $ for player 2. We denote this specific problem as $(\text{DTSDG})^{0}$. In this case, the state equation is expressed as:
\begin{equation}\label{eq:2.8}
    \left\{
    \begin{aligned}
        x^{(\mathbf{u}, \mathbf{v})}_{k+1} &= A_k x^{(\mathbf{u}, \mathbf{v})}_k + B_k u_k + C_k v_k \\
        & \quad + \left[ D_k x^{(\mathbf{u}, \mathbf{v})}_k + E_k u_k + F_k v_k \right] \omega_k, \\
        x_0^{(\mathbf{u}, \mathbf{v})} &= 0 \in \mathbb{R}^n, \quad k \in \mathcal{T}.
    \end{aligned}
    \right.
\end{equation}

The corresponding cost functionals for player 1 and player 2 in this scenario are defined as:
\begin{equation}\label{eq:2.9}
    \begin{split}
        \mathcal{J}_1^{0}(0, &0; \mathbf{u}, \mathbf{v}) \\=& \frac{1}{2} \mathbb{E} \bigg\{  \langle G_N x^{(\mathbf{u}, \mathbf{v})}_N, x^{(\mathbf{u}, \mathbf{v})}_N \rangle  + \sum_{k=0}^{N-1} \Big[ \langle Q_k x^{(\mathbf{u}, \mathbf{v})}_k, x^{(\mathbf{u}, \mathbf{v})}_k \rangle \\
        & + 2 \langle L_k^{\top} u_k, x^{(\mathbf{u}, \mathbf{v})}_k \rangle + \langle R_k u_k, u_k \rangle \Big] \bigg\},
    \end{split}
\end{equation}
and
\begin{equation}\label{eq:2.10}
    \begin{split}
        \mathcal{J}_2^{0}(0, &0; \mathbf{u}, \mathbf{v}) \\=& \frac{1}{2} \mathbb{E} \bigg\{  \langle H_N x^{(\mathbf{u}, \mathbf{v})}_N, x^{(\mathbf{u}, \mathbf{v})}_N \rangle  + \sum_{k=0}^{N-1} \Big[ \langle P_k x^{(\mathbf{u}, \mathbf{v})}_k, x^{(\mathbf{u}, \mathbf{v})}_k \rangle \\
        & + 2 \langle M_k^{\top} v_k, x^{(\mathbf{u}, \mathbf{v})}_k \rangle + \langle S_k v_k, v_k \rangle \Big] \bigg\}.
    \end{split}
\end{equation}

\section{Variational Inequality for Nash Equilibrium in Stochastic Games}

This section establishes a rigorous mathematical framework for deriving the variational inequality that characterizes the Nash equilibrium in two-player stochastic difference games. We will first discuss the $G\hat{a}teaux$ differentiability of cost functionals. Our primary focus will then be on applying this concept to derive the necessary and sufficient conditions for the existence of the Nash equilibrium point of Problem 2.2, specifically the variational inequality characterization. To support our analysis, we will introduce several fundamental lemmas that will guide us to the desired variational inequality.

\begin{lem}\label{lem:3.1}
    Let Assumptions \ref{ass:2.1}-\ref{ass:2.3} hold. Let $ \mathbf{v} \in\mathcal{V}$ be any given admissible control.
    Then for any admissible controls $\mathbf{u}\in \mathcal{U}$, $\mathbf{w}\in \mathcal{U}$,
    $\mathcal{J}_1(0, \xi; \mathbf{u}, \mathbf{v})$ is $G\hat{a}teaux$ differentiable at $\mathbf{u}$
    and the corresponding $G\hat{a}teaux$ derivative $\mathcal{J}_1'(0, \xi; \mathbf{u}, \mathbf{v})$ at $\mathbf{u}$ is given by
    \begin{equation}\label{eq:3.1}
        \begin{split}
            \big\langle &\mathcal{J}_1'(0, \xi; \mathbf{u}, \mathbf{v}), \mathbf{w} \big\rangle
            \\=& \mathbb{E} \bigg[ \big\langle G_{N}x_{N} + g_{N}, x^{(\mathbf{w}, \mathbf{0})}_{N} \big\rangle + \sum^{N-1}_{k=0} \bigg( \big\langle Q_{k}x_{k} + L^{\top}_{k}u_{k} + q_{k}, x^{(\mathbf{w}, \mathbf{0})}_{k} \big\rangle \\
            &+ \big\langle L_{k}x_{k} + R_{k}u_{k} + \rho_{k}, w_{k} \big\rangle \bigg) \bigg],
        \end{split}
    \end{equation}
    where $x(\cdot)$ is the solution of \eqref{eq:2.1} corresponding to the admissible control $(\mathbf{u}$,$\mathbf{v})$
    with the initial state $x_{0} = \xi$; $x^{(\mathbf{w}, \mathbf{0})}$ is the solution of \eqref{eq:2.8} corresponding to
    the admissible control with $(\mathbf{u}$,$\mathbf{v})$ =$(\mathbf{w},0)$.
\end{lem}

\begin{proof}
    We consider the cost functional for player 1 given by \eqref{eq:2.4}
    where $x_k$ is the state process satisfying the system:
    \begin{equation}\label{eq:3.2}
        \begin{split}
            x_{k+1} = &A_{k} x_{k} + B_{k} u_{k} + C_{k} v_{k} + b_{k} \\
            &+ \big[D_{k} x_{k} + E_{k} u_{k} + F_{k} v_{k} + \sigma_{k}\big] \omega_{k},
        \end{split}
    \end{equation}
    with $x_0 = \xi$. Consider perturbing the control $\mathbf{u}$ by a small parameter $\epsilon>0$
    such that the perturbed control is $\mathbf{u} + \epsilon \mathbf{w}$. Then the state process
    corresponding to the admissible pair $(\mathbf{u} + \epsilon \mathbf{w},\mathbf{v})$ denoted by $x^{\epsilon}(\cdot)$ satisfies
    \begin{equation}\label{eq:3.4}
        \begin{split}
            x_{k+1}^{\epsilon} = &A_{k} x_{k}^{\epsilon} + B_{k} (u_{k} + \epsilon w_{k}) + C_{k} v_{k} + b_{k} \\
            &+ \big[D_{k} x_{k}^{\epsilon} + E_{k} (u_{k} + \epsilon w_{k}) + F_{k} v_{k} + \sigma_{k}\big] \omega_{k}.
        \end{split}
    \end{equation}
    The corresponding cost functional is given by
    \begin{equation}\label{eq:3.3}
        \mathcal{J}_1(0, \xi; \mathbf{u} + \epsilon \mathbf{w}, \mathbf{v}).
    \end{equation}
    From \eqref{eq:2.8}, we get that $x^{(\mathbf{w}, \mathbf{0})}$ satisfies
    \begin{equation}\label{eq:3.6}
        \begin{split}
            x_{k+1}^{(\mathbf{w}, \mathbf{0})} = A_{k} x_{k}^{(\mathbf{w}, \mathbf{0})} + B_{k} w_{k} + (D_{k} x_{k}^{(\mathbf{w}, \mathbf{0})} + E_{k} w_{k} ) \omega_{k}.
        \end{split}
    \end{equation}
    Then from \eqref{eq:3.6}, \eqref{eq:3.2} and \eqref{eq:3.4}, it is easy to check that
    \begin{equation}\label{eq:3.7}
        x_{k+1}^{\epsilon} = x_{k+1} + \epsilon x_{k+1}^{(\mathbf{w}, \mathbf{0})}
    \end{equation}
    Expanding $\mathcal{J}_1(0, \xi; \mathbf{u} + \epsilon \mathbf{w}, \mathbf{v})$, we have
    \small{
    \begin{equation}\label{eq:3.8}
        \begin{split}
            \mathcal{J}_1&(0, \xi; \mathbf{u} + \epsilon \mathbf{w}, \mathbf{v})
            \\= &\mathcal{J}_1(0, \xi; \mathbf{u}, \mathbf{v}) + \frac{1}{2}\epsilon^2\mathcal{J}_1^{0}(0,0;\mathbf{w},\mathbf{0}) + \epsilon \mathbb{E}\bigg[ \langle G_{N} x_{N} + g_{N}, x^{(\mathbf{w}, \mathbf{0})}_{N} \rangle \\
            &+ \sum_{k=0}^{N-1} \Big( \langle Q_{k} x_{k} + L_{k}^{\top} u_{k} + q_{k}, x^{(\mathbf{w}, \mathbf{0})}_{k} \rangle + \langle L_{k} x_{k} + R_{k} u_{k} + \rho_{k}, w_{k} \rangle \Big) \bigg],
        \end{split}
    \end{equation}}
\normalsize
    where $\mathcal{J}_1^{0}(0,0;\mathbf{w},\mathbf{0})$ is defined by \eqref{eq:2.9} with $\mathbf{v}=0$. Therefore,
    \begin{equation}\label{eq:3.9}
        \begin{split}
            \lim_{\epsilon \to 0^{+}}& \frac{\mathcal{J}_1(0, \xi; \mathbf{u} + \epsilon \mathbf{w}, \mathbf{v}) - \mathcal{J}_1(0, \xi; \mathbf{u}, \mathbf{v})}{\epsilon} \\
            = &\mathbb{E} \bigg[ \langle G_{N} x_{N} + g_{N}, x^{(\mathbf{w}, \mathbf{0})}_{N} \rangle + \sum_{k=0}^{N-1} \Big( \langle Q_{k} x_{k} + L_{k}^{\top} u_{k} + q_{k}, x^{(\mathbf{w}, \mathbf{0})}_{k} \rangle \\
            &+ \langle L_{k} x_{k} + R_{k} u_{k} + \rho_{k}, w_{k} \rangle \Big) \bigg].
        \end{split}
    \end{equation}
    As a result, $\mathcal{J}_1(0, \xi; \mathbf{u}, \mathbf{v})$ is $G\hat{a}teaux$ differentiable
    at $\mathbf{u}$ and the corresponding $G\hat{a}teaux$ derivative $\mathcal{J}_1'(0, \xi; \mathbf{u}, \mathbf{v})$ at $\mathbf{u}$ is given by \eqref{eq:3.1}.
    The proof is complete.
\end{proof}

\begin{lem}\label{lem:3.2}
    Let Assumptions \ref{ass:2.1}-\ref{ass:2.3} hold. Then for any admissible controls
    $(\mathbf{u}, \mathbf{v} )\in \mathcal{U}\times\mathcal{V}$, $\mathbf{z}\in \mathcal{V}$,
    $\mathcal{J}_2(0, \xi; \mathbf{u}, \mathbf{v})$ is $G\hat{a}teaux$ differentiable at $\mathbf{v}$
    and the corresponding $G\hat{a}teaux$ derivative $\mathcal{J}_2'(0, \xi; \mathbf{u}, \mathbf{v})$ is given by
    \begin{equation}\label{eq:3.10}
        \begin{split}
            \big\langle \mathcal{J}&_2'(0, \xi; \mathbf{u}, \mathbf{v}), \mathbf{z}\big\rangle
            \\=& \mathbb{E} \bigg[ \big\langle H_{N}x_{N} + h_{N}, x^{(\mathbf{0}, \mathbf{z})}_{N} \big\rangle + \sum^{N-1}_{k=0} \bigg( \big\langle P_{k}x_{k} + M^{\top}_{k}v_{k} + p_{k}, x^{(\mathbf{0}, \mathbf{z})}_{k} \big\rangle \\
            &+ \big\langle M_{k}x_{k} + S_{k}v_{k} + \theta_{k}, z_{k} \big\rangle \bigg) \bigg],
        \end{split}
    \end{equation}
    where $x(\cdot)$ is the solution of \eqref{eq:2.1} corresponding to the admissible control $(\mathbf{u}$,$\mathbf{v})$
    with the initial state $x_{0} = \xi$; $x^{(\mathbf{0}, \mathbf{z})}$ is the solution of \eqref{eq:2.8} corresponding to
    the admissible control with $(\mathbf{u}$,$\mathbf{v})$ =$(\mathbf{0},\mathbf z)$.
\end{lem}

\begin{proof}
    We have proved that $\mathcal{J}_1(0, \xi; \mathbf{u}, \mathbf{v})$ is $G\hat{a}teaux$-differentiable in Lemma \ref{lem:3.1}.
    Similarly, we can show that $\mathcal{J}_2(0, \xi; \mathbf{u}, \mathbf{v})$ is also $G\hat{a}teaux$-differentiable.
\end{proof}

\begin{thm}\label{thm:3.3}
    Let Assumptions \ref{ass:2.1}-\ref{ass:2.3} hold. Then a necessary and sufficient condition for
    $(\mathbf{u}^{\ast}, \mathbf{v}^{\ast}) \in \mathcal{U} \times \mathcal{V}$ to be a Nash equilibrium point
    of Problem \ref{pro:2.2} is that
    \begin{equation}\label{eq:3.11}
        \langle \mathcal{J}_1'(0, \xi; \mathbf{u}^{\ast}, \mathbf{v}^{\ast}), \mathbf{w}\rangle = 0,~~ \forall \mathbf{w} \in \mathcal{U},
    \end{equation}
    and
    \begin{equation}\label{eq:3.12}
        \langle \mathcal{J}_2'(0, \xi; \mathbf{u}^{\ast}, \mathbf{v}^{\ast}),\mathbf{z} \rangle = 0,~~ \forall \mathbf{z} \in \mathcal{V}.
    \end{equation}
\end{thm}

\begin{proof}
    $\mathbf{(Sufficiency):}$
    For player 1, assume there exists an admissible control $\mathbf{u}^{\ast}\in \mathcal{U}$ such that
    \begin{equation}\label{eq:3.13}
        \langle \mathcal{J}_1'(0, \xi; \mathbf{u}^{\ast}, \mathbf{v}^{\ast}), \mathbf{w}\rangle = 0,~~ \forall \mathbf{w} \in \mathcal{U}.
    \end{equation}
    We need to show that $\mathbf{u}^{\ast}$ is an optimal control for player 1 given $\mathbf{v}^{\ast}$.
    For any admissible control $\mathbf{u}\in \mathcal{U}$, we need to verify:
    $$
    \mathcal{J}_1(0, \xi; \mathbf{u}, \mathbf{v}^{\ast}) - \mathcal{J}_1(0, \xi; \mathbf{u}^{\ast}, \mathbf{v}^{\ast}) \geq 0.
    $$
    Considering $\epsilon=1$, $\mathbf{u}=\mathbf{u}^\ast$, $\mathbf{w}= \mathbf{u}-\mathbf{u}^\ast$ in \eqref{eq:3.8} and \eqref{eq:3.13},
    \begin{equation}\label{eq:3.14}
        \begin{split}
            &\mathcal{J}_1(0, \xi; \mathbf{u}, \mathbf{v}^{\ast}) - \mathcal{J}_1(0, \xi; \mathbf{u}^{\ast}, \mathbf{v}^{\ast}) \\
            &= \mathcal{J}_1(0, \xi;  \mathbf{u}^{\ast}+\mathbf{u}-\mathbf{u}^\ast, \mathbf{v}^{\ast}) - \mathcal{J}_1(0, \xi; \mathbf{u}^{\ast}, \mathbf{v}^{\ast}) \\
            &= \frac{1}{2}\mathcal{J}_1^0(0, 0; \mathbf{u} - \mathbf{u}^{\ast}, 0) + \langle \mathcal{J}_1'(0, \xi; \mathbf{u}^{\ast}, \mathbf{v}^{\ast}), \mathbf{u} - \mathbf{u}^{\ast}  \rangle \\
            &\geq 0 + \langle \mathcal{J}_1'(0, \xi; \mathbf{u}^{\ast}, \mathbf{v}^{\ast}), \mathbf{u} - \mathbf{u}^{\ast} \rangle \\
            &= 0.
        \end{split}
    \end{equation}
    This implies that $\mathbf{u}^{\ast}$ is optimal for player 1 given $\mathbf{v}^{\ast}$.

    Similarly for player 2, we can prove $\mathcal{J}_2(0, \xi; \mathbf{u}^{\ast}, \mathbf{v}) - \mathcal{J}_2(0, \xi; \mathbf{u}^{\ast}, \mathbf{v}^{\ast}) \geq 0$,
    which implies that $\mathbf{v}^{\ast}$ is optimal for player 2 given $\mathbf{u}^{\ast}$.
    Therefore, $(\mathbf{u}^{\ast}, \mathbf{v}^{\ast})$ is a Nash equilibrium point and sufficiency is proved.

    $\mathbf{(Necessity):}$
    Assume $(\mathbf{u}^{\ast}, \mathbf{v}^{\ast})$ is a Nash equilibrium point. Then for player 1, $\mathbf{u}^{\ast}$ is optimal
    for the cost functional $\mathcal{J}_1(0, \xi; \mathbf{u}, \mathbf{v})$ given $\mathbf{v}^{\ast}$, i.e.
    $$
    \mathcal{J}_1(0, \xi; \mathbf{u}^{\ast}, \mathbf{v}^{\ast}) = \min _{\mathbf{u} \in \mathcal{U}}\mathcal{J}_1(0, \xi; \mathbf{u}, \mathbf{v}^{\ast}).
    $$
    For any admissible $\mathbf{w}\in\mathcal{U}$ and $\lambda>0$,
    $$
    \mathcal{J}_1(0,\xi;\mathbf{u}^{\ast}+\lambda \mathbf{w},\mathbf{v}^{\ast})-\mathcal{J}_1(0,\xi; \mathbf{u}^{\ast},\mathbf{v}^{\ast})\geq 0,
    $$
    $$
    \mathcal{J}_1(0,\xi;\mathbf{u}^{*}-\lambda \mathbf{w},\mathbf{v}^{\ast})-\mathcal{J}_1(0,\xi; \mathbf{u}^{\ast},\mathbf{v}^{\ast})\geq 0.
    $$
    From the $G\hat{a}teaux$ derivative definition:
   {\small \begin{equation}
\begin{split}
\big\langle &\mathcal{J}'_1(0,\xi; \mathbf{u}^{\ast},\mathbf{v}^{\ast}),\mathbf{w} \big\rangle
\\ =& \lim_{\lambda\to 0^{+}} \frac{
    \mathcal{J}_1(0,\xi; \mathbf{u}^{\ast}+\lambda \mathbf{w},\mathbf{v}^{\ast})
    - \mathcal{J}_1(0,\xi; \mathbf{u}^{\ast},\mathbf{v}^{\ast})
}{\lambda} \\
 \geq& 0, \\[2ex]
\big\langle &\mathcal{J}'_1(0,\xi; \mathbf{u}^{\ast},\mathbf{v}^{\ast}), -\mathbf{w} \big\rangle
\\=& \lim_{\lambda\to 0^{+}} \frac{
    \mathcal{J}_1(0,\xi; \mathbf{u}^{\ast}-\lambda\mathbf{w},\mathbf{v}^{\ast})
    - \mathcal{J}_1(0,\xi; \mathbf{u}^{*},\mathbf{v}^{\ast})
}{\lambda} \\
 \geq & 0.
\end{split}
\end{equation}}
    By linearity:
    $$
    0 \leq \langle \mathcal{J}'_1(0,\xi; \mathbf{u}^{\ast},\mathbf{v}^{\ast}),\mathbf{w} \rangle = -\langle \mathcal{J}'_1(0,\xi; \mathbf{u}^{\ast},\mathbf{v}^{\ast}), -\mathbf{w}\rangle \leq 0,
    $$
    which yields \eqref{eq:3.11}.

    Similarly for player 2:
    $$
    \langle \mathcal{J}_2'(0,\xi; \mathbf{u}^{\ast},\mathbf{v}^{\ast}), \mathbf{z} \rangle = 0, \quad \forall \mathbf{z} \in \mathcal{V}.
    $$
    The proof is complete.
\end{proof}

\section{Stochastic Hamiltonian Systems and Stochastic Riccati equations }
	
This section will present two Hamiltonian systems for discrete-time stochastic non-zero-sum difference games involving two players, aiming to derive an explicit expression for the open-loop Nash equilibrium point. The stochastic Hamiltonian systems, governed by stability conditions, result in fully coupled forward-backward stochastic difference equations, which are particularly challenging to solve. This complexity arises from the interdependence of the forward and backward dynamics, necessitating a more sophisticated approach to derive solutions effectively. To address this, we employ stochastic Riccati difference equations to decouple the systems, allowing us to derive feedback representations of the Nash equilibrium points. This approach enhances the efficiency of solving discrete-time games. Next, let's start with an introduction to the stochastic Hamiltonian system.

\subsection{Stationary Conditions and Hamiltonian System}

We now introduce the Hamiltonian functions associated with discrete-time stochastic difference non-zero-sum games. For any $ x \in L_{\mathbb{F}}^2(\overline{\mathcal{T}}; \mathbb{R}^n) $, $ y_1, y_2 \in L_{\mathbb{F}}^2(\overline{\mathcal{T}}; \mathbb{R}^n) $, $ u \in \mathcal{U} $, and $ v \in \mathcal{V} $, for each $ k \in \mathcal{T} $, the Hamiltonian functions for player 1 and player 2 are defined, respectively, as

\begin{equation}\label{eq:4.1}
	\begin{split}
		H_1(k, x, u, v, y_1) &:= \mathbb{E}\bigg[\langle y_{1,k+1}, A_k x_k + B_k u_k + C_k v_k \rangle \\
		&\quad + \langle y_{1,k+1} \omega_k, D_k x_k + E_k u_k + F_k v_k \rangle \bigm| \mathcal{F}_{k-1}\bigg] \\
		&\quad + \frac{1}{2} \langle Q_k x_k, x_k \rangle + \frac{1}{2} \langle R_k u_k, u_k \rangle \\
		&\quad + \langle L_k^{\top} u_k, x_k \rangle + \langle q_k, x_k \rangle + \langle \rho_k, u_k \rangle.
	\end{split}
\end{equation}
and
\begin{equation}\label{eq:4.2}
	\begin{split}
		H_2(k, x, u, v, y_2) &:= \mathbb{E}\bigg[\langle y_{2,k+1}, A_k x_k + B_k u_k + C_k v_k \rangle \\
		&\quad + \langle y_{2,k+1} \omega_k, D_k x_k + E_k u_k + F_k v_k \rangle \bigm| \mathcal{F}_{k-1}\bigg] \\
		&\quad + \frac{1}{2} \langle P_k x_k, x_k \rangle + \frac{1}{2} \langle S_k v_k, v_k \rangle \\
		&\quad + \langle M_k^{\top} v_k, x_k \rangle + \langle p_k, x_k \rangle + \langle \sigma_k, v_k \rangle.
	\end{split}
\end{equation}

For player 1, we denote the partial derivatives of $ H_1 $ with respect to the state $ x $ and control $ u $ as $ \partial_x H_1(k, x, u, v, y_1) $ and $ \partial_u H_1(k, x, u, v, y_1) $, respectively. Similarly, for player 2, $ \partial_x H_2(k, x, u, v, y_2) $ and $ \partial_v H_2(k, x, u, v, y_2) $ represent the partial derivatives of $ H_2 $ with respect to $ x $ and $ v $.

The adjoint processes corresponding to an admissible triple $ (x, \mathbf{u}, \mathbf{v}) $ for players 1 and 2 are denoted by $ y_1 = \{y_{1,k}\}_{k=0}^{N} $ and $ y_2 = \{y_{2,k}\}_{k=0}^{N} $, respectively. These processes satisfy the following backward stochastic difference equations (BS$\Delta$E):

\begin{equation}\label{eq:4.3}
\begin{split}
&\left\{
\begin{aligned}
y_{1,k} &= \partial_x H_1(k, x, u, v, y_1) \\
&= A_k^\top \mathbb{E}[y_{1,k+1} \mid \mathcal{F}_{k-1}] + D_k^\top \mathbb{E}[y_{1,k+1} \omega_k \mid \mathcal{F}_{k-1}] \\
&\quad + Q_k x_k + L_k^\top u_k + q_k, \\
y_{1,N} &= G_N x_N + g_N, \quad k \in \mathcal{T},
\end{aligned}
\right.
\end{split}
\end{equation}

\begin{equation}\label{eq:4.4}
\begin{split}
&\left\{
\begin{aligned}
y_{2,k} &= \partial_x H_2(k, x, u, v, y_2) \\
&= A_k^\top \mathbb{E}[y_{2,k+1} \mid \mathcal{F}_{k-1}] + D_k^\top \mathbb{E}[y_{2,k+1} \omega_k \mid \mathcal{F}_{k-1}] \\
&\quad + P_k x_k + M_k^\top v_k + p_k, \\
y_{2,N} &= H_N x_N + h_N, \quad k \in \mathcal{T}.
\end{aligned}
\right.
\end{split}
\end{equation}
The above BS$\Delta$E, when combined with the state equations, forms a FBS$\Delta$E, known as the stochastic Hamiltonian system, that provides a dual explicit characterization of the Nash equilibrium point $ (\mathbf{u}^*, \mathbf{v}^*) $. This coupling highlights the interdependencies between the players' strategies and their respective state and adjoint processes. By analyzing these equations, we can derive the necessary stationary conditions that the Nash equilibrium must satisfy in the following.

\begin{thm}\label{thm:4.1}
\textbf{(Stationarity Conditions)} Let Assumptions $\ref{ass:2.1}$-$\ref{ass:2.3}$ hold. Then, the necessary and sufficient condition for an admissible control pair $( \mathbf{u}, \mathbf{v}) \in \mathcal{U} \times \mathcal{V}$ to be an open-loop Nash equilibrium point of Problem $\ref{pro:2.2}$, with the state equation given by $ x $, is that the following stationarity conditions are satisfied:
\begin{equation}\label{eq:4.5}
    \partial_u H_1(k, x, u, v, y_1) = 0,
\end{equation}
and
\begin{equation}\label{eq:4.6}
    \partial_v H_2(k, x, u, v, y_2) = 0.
\end{equation}
That is, the following conditions must hold for all $k \in \mathcal{T}$:
\begin{equation}\label{eq:4.7}
\begin{split}
&B_k^\top \mathbb{E}[y_{1, k+1} \mid \mathcal{F}_{k-1}] + E_k^\top \mathbb{E}[y_{1, k+1} \omega_k \mid \mathcal{F}_{k-1}] \\
&+ L_k x_k + R_k u_k + \rho_k = 0,
\end{split}
\end{equation}
and
\begin{equation}\label{eq:4.8}
\begin{split}
&C_k^\top \mathbb{E}[y_{2, k+1} \mid \mathcal{F}_{k-1}] + F_k^\top \mathbb{E}[y_{2, k+1} \omega_k \mid \mathcal{F}_{k-1}] \\
&+ M_k x_k + S_k v_k + \theta_k = 0.
\end{split}
\end{equation}
Here, $ y_1 = \{y_{1,k}\}_{k=0}^{N} \in L_{\mathbb{F}}^2(\overline{\mathcal{T}}; \mathbb{R}^n) $ and $ y_2 = \{y_{2,k}\}_{k=0}^{N} \in L_{\mathbb{F}}^2(\overline{\mathcal{T}}; \mathbb{R}^n) $, as defined by equations \eqref{eq:4.3} and \eqref{eq:4.4}, are the solutions to the adjoint equations corresponding to the admissible 3-tuple $ (x, \mathbf{u}, \mathbf{v}) $.
\end{thm}

\begin{proof}
$\mathbf{Necessary ~Condition:}$~
Suppose that the admissible control pair $( \mathbf{u}, \mathbf{v})$ is an open-loop Nash-equilibrium point. Under Assumptions $\ref{ass:2.1}$-$\ref{ass:2.3}$, the BS$\Delta$E \eqref{eq:4.3} admit a unique solution $y_1 \in L_{\mathbb{F}}^2(\overline{\mathcal{T}}; \mathbb{R}^n)$. For any admissible control $\mathbf{w}^* \in \mathcal{U}$, consider the variational process $x^{(\mathbf{w}^*,0)}$ governed by \eqref{eq:2.8}.

Summing $\langle y_{1,k+1}, x^{(\mathbf{w}^*,0)}_{k+1} \rangle - \langle y_{1,k}, x^{(\mathbf{w}^*,0)}_{k} \rangle$ from $k=0$ to $N-1$ and taking expectations:
\begin{equation}\label{eq:4.9-revised}
\begin{split}
&\mathbb{E}\big[\langle y_{1,N}, x^{(\mathbf{w}^*,0)}_{N} \rangle\big]
\\&= \mathbb{E}\sum_{k=0}^{N-1} \Big[ \langle A_k^\top y_{1,k+1} + D_k^\top y_{1,k+1}\omega_k - y_{1,k}, x^{(\mathbf{w}^*,0)}_{k} \rangle \\
&\quad + \langle B_k^\top y_{1,k+1} + E_k^\top y_{1,k+1}\omega_k, w^*_k \rangle \Big] \\
&= \mathbb{E}\sum_{k=0}^{N-1} \Big[ -\langle Q_k x_k + L_k^\top u_k + q_k, x^{(\mathbf{w}^*,0)}_{k} \rangle \\
&\quad + \langle B_k^\top y_{1,k+1} + E_k^\top y_{1,k+1}\omega_k, w^*_k \rangle \Big].
\end{split}
\end{equation}

Combining with the cost functional's structure:
\begin{eqnarray}\label{eq:4.10-revised}
    \begin{split}
        &\mathbb{E}\big[\langle G_N x_N + g_N, x^{(\mathbf{w}^*,0)}_N \rangle\big] + \mathbb{E}\sum_{k=0}^{N-1} \langle Q_k x_k + L_k^\top u_k + q_k, x^{(\mathbf{w}^*,0)}_k \rangle \\
        &= \mathbb{E}\sum_{k=0}^{N-1} \Big\langle B_k^\top \mathbb{E}[y_{1,k+1} \mid \mathcal{F}_{k-1}] + E_k^\top \mathbb{E}[y_{1,k+1}\omega_k \mid \mathcal{F}_{k-1}], w^*_k \Big\rangle.
    \end{split}
\end{eqnarray}
Here, we utilized  $\mathbf{w}^*$ is adapted: $w^*_k \in \mathcal{F}_{k-1}$.

Substituting  \eqref {eq:4.10-revised} into  G$\hat{a}$teaux derivative \eqref{eq:3.1} in Lemma \ref{lem:3.1}  and using  Theorem \ref {thm:3.3}, we get that
\begin{equation}\label{eq:4.11-revised}
	\begin{split}
		\langle \mathcal{J}_1'&(0, \xi; \mathbf{u}, \mathbf{v}), \mathbf{w}^* \rangle
		\\=& \mathbb{E}\sum_{k=0}^{N-1} \bigg\langle
		B_k^\top \mathbb{E}[y_{1,k+1} \mid \mathcal{F}_{k-1}] + E_k^\top \mathbb{E}[y_{1,k+1}\omega_k \mid \mathcal{F}_{k-1}] \\
		&+ L_k x_k + R_k u_k + \rho_k,\ w^*_k \bigg\rangle = 0.
	\end{split}
\end{equation}
The arbitrariness of $\mathbf{w}^*$ implies \eqref{eq:4.7}. Similarly, \eqref{eq:4.8} follows for player 2. Necessity is proved.

\textbf{Sufficient Condition:}
Let $(x, \mathbf{u}, \mathbf{v})$ be an admissible 3-tuple for Problem~\ref{pro:2.2}, and let $y_1 = \{y_{1,k}\}_{k=0}^N$ and $y_2 = \{y_{2,k}\}_{k=0}^N$ denote the solutions of the adjoint equations~\eqref{eq:4.3} and~\eqref{eq:4.4}, respectively. Suppose the stationarity conditions~\eqref{eq:4.7}-\eqref{eq:4.8} hold.

Substituting~\eqref{eq:4.7} into the G$\hat{a}$teaux derivative expression~\eqref{eq:4.11-revised}, we obtain:
\begin{equation}
\langle \mathcal{J}_1'(0, \xi; \mathbf{u}, \mathbf{v}), \mathbf{w}^* \rangle = 0, \quad \forall \mathbf{w}^* \in \mathcal{U}.
\end{equation}
Similarly for player 2, substituting~\eqref{eq:4.8} into the corresponding derivative yields:
\begin{equation}
\langle \mathcal{J}_2'(0, \xi; \mathbf{u}, \mathbf{v}), \mathbf{z}^* \rangle = 0, \quad \forall \mathbf{z}^* \in \mathcal{V}.
\end{equation}

By Theorem~\ref{thm:3.3}, these vanishing G$\hat{a}$teaux derivatives imply that $(\mathbf{u}, \mathbf{v})$ simultaneously minimizes $\mathcal{J}_1$ and $\mathcal{J}_2$ under mutual optimality. Therefore, $(x, \mathbf{u}, \mathbf{v})$ constitutes an open-loop Nash equilibrium for Problem~\ref{pro:2.2}. The proof is complete.
\end{proof}

\begin{rmk}\label{rmk:4.2}
    Under Assumptions \ref{ass:2.1}-\ref{ass:2.3}, any admissible 3-tuple $(x, \mathbf{u}, \mathbf{v})$ satisfies:
    \begin{equation} \label{eq:4.14}
    	\begin{split}
    		u_k =& -R_{k}^{-1}\big( B_{k}^{\top}\mathbb{E}\big[y_{1,k+1} \big| \mathcal{F}_{k-1}\big] + E_{k}^{\top}\mathbb{E}\big[y_{1,k+1}\omega_k \big| \mathcal{F}_{k-1}\big] \\&+ L_k x_k + \rho_k \big)
    	\end{split}
    \end{equation}
   \begin{equation} \label{eq:4.15}
   	\begin{split}
        v_k =& -S_k^{-1}\big( C_k^{\top}\mathbb{E}\big[y_{2,k+1} \big| \mathcal{F}_{k-1}\big] + F_k^{\top}\mathbb{E}\big[y_{2,k+1}\omega_k \big| \mathcal{F}_{k-1}\big] \\&+ M_k x_k + \theta_k \big)
    \end{split}
   \end{equation}
    with the first-order conditions equivalently expressed as:
    \begin{equation} \label{eq:4.16}
        \mathbb{E}\big[ B_k^{\top} y_{1,k+1} + E_k^{\top} y_{1,k+1}\omega_k \big| \mathcal{F}_{k-1} \big] + L_k x_k + R_k u_k + \rho_k = 0  \end{equation}
         \begin{equation} \label{eq:4.17}
         	\mathbb{E}\big[ C_k^{\top} y_{2,k+1} + F_k^{\top} y_{2,k+1}\omega_k \big| \mathcal{F}_{k-1} \big] + M_k x_k + S_k v_k + \theta_k = 0.
    \end{equation}
    for all $k \in \mathcal{T}$.
\end{rmk}

\begin{equation}\label{eq:4.19}
	\left\{
	\begin{aligned}
		x_{k+1} &= A_{k}x_{k} + B_{k}u_{k} + C_{k}v_{k} + b_{k} + \big(D_{k}x_{k} + E_{k}u_{k} \\
		&\quad + F_{k}v_{k} + \sigma_{k}\big)\omega_{k}, \\
		y_{1, k} &= A_{k}^{\top}\mathbb{E}[y_{1, k+1} \mid \mathcal{F}_{k-1}] + D_{k}^{\top}\mathbb{E}[y_{1, k+1}\omega_{k} \mid \mathcal{F}_{k-1}] \\
		&\quad + Q_{k}x_{k} + L_{k}^{\top}u_{k} + q_{k}, \\
		x_{0} &= \xi \in \mathbb{R}^n, \quad y_{1, N} = G_{N}x_{N} + g_{N}, \\
		0 &= B_{k}^{\top}\mathbb{E}[y_{1, k+1} \mid \mathcal{F}_{k-1}] + E_{k}^{\top}\mathbb{E}[y_{1, k+1}\omega_{k} \mid \mathcal{F}_{k-1}] \\
		&\quad + L_{k}x_{k} + R_{k}u_{k} + \rho_{k}, \quad k \in \mathcal{T}.
	\end{aligned}
	\right.
\end{equation}

\begin{equation}\label{eq:4.20}
	\left\{
	\begin{aligned}
		x_{k+1} &= A_{k}x_{k} + B_{k}u_{k} + C_{k}v_{k} + b_{k} + \big(D_{k}x_{k} + E_{k}u_{k} \\
		&\quad + F_{k}v_{k} + \sigma_{k}\big)\omega_{k}, \\
		y_{2,k} &= A_{k}^{\top}\mathbb{E}\big[y_{2,k+1} \mid \mathcal{F}_{k-1}\big] + D_{k}^{\top}\mathbb{E}\big[y_{2,k+1}\omega_{k} \mid \mathcal{F}_{k-1}\big] \\
		&\quad + P_{k}x_{k} + M_{k}^{\top}v_{k} + p_{k}, \\
		x_{0} &= \xi \in \mathbb{R}^n, \quad
		y_{2,N} = H_{N}x_{N} + h_{N}, \\
		0 &= C_{k}^{\top}\mathbb{E}\big[y_{2,k+1} \mid \mathcal{F}_{k-1}\big] + F_{k}^{\top}\mathbb{E}\big[y_{2,k+1}\omega_{k} \mid \mathcal{F}_{k-1}\big] \\
		&\quad + M_{k}x_{k} + S_{k}v_{k} + \theta_{k}, \quad k \in \mathcal{T}.
	\end{aligned}
	\right.
\end{equation}

The Hamiltonian systems \eqref{eq:4.19} and \eqref{eq:4.20} form two fully coupled forward-backward stochastic difference equations (FBS$\Delta$Es). Under Assumptions \ref{ass:2.1}--\ref{ass:2.3} and based on Theorem~\ref{thm:3.3}, any Nash equilibrium $(\mathbf{u}, \mathbf{v})$ must correspond to a solution
$$
(\mathbf{u}, x, y_1) \in \mathcal{U} \times L_{\mathbb{F}}^2(\overline{\mathcal{T}}; \mathbb{R}^n) \times L_{\mathbb{F}}^2(\overline{\mathcal{T}}; \mathbb{R}^n)
$$
of \eqref{eq:4.19}, and
$$
(\mathbf{v}, x, y_2) \in \mathcal{V} \times L_{\mathbb{F}}^2(\overline{\mathcal{T}}; \mathbb{R}^n) \times L_{\mathbb{F}}^2(\overline{\mathcal{T}}; \mathbb{R}^n)
$$
of \eqref{eq:4.20}, respectively. Conversely, if the Hamiltonian systems \eqref{eq:4.19} and \eqref{eq:4.20} admit a common state process $x$ and control processes $(\mathbf{u}, \mathbf{v})$ satisfying both systems, then $(\mathbf{u}, \mathbf{v})$ constitutes a Nash equilibrium.

\subsection{Riccati Equations and Feedback Representation of Optimal Controls}
	
Since the Hamiltonian systems \eqref{eq:4.19} and \eqref{eq:4.20} are fully coupled FBS$\Delta$E, solving them directly poses significant analytical challenges. To decouple these systems and derive the feedback representation of the open-loop Nash equilibrium, we introduce an appropriate set of Riccati-type difference equations to simplify the analysis. In the following, we formally derive the associated Riccati equations, using the method of undetermined coefficients as a guiding technique. To this end, inspired by the terminal conditions of systems \eqref{eq:4.19} and \eqref{eq:4.20}, we postulate specific affine relationships between the state process $x$  associated with the open-loop Nash equilibrium pair
	$(\mathbf u, \mathbf v)$ and the corresponding adjoint processes \( y_1 = \{y_{1, k}\}_{k=0}^{N} \) and \( y_2 = \{y_{2, k}\}_{k=0}^{N} \) :
	\begin{equation}\label{eq:4.21}
		y_{1,k}=T^{1}_{k}x_{k}+\phi^{1}_{k},~~~~~~~~y_{2,k}=T^{2}_{k}x_{k}+\phi^{2}_{k},
	\end{equation}
	where  $T^{1}:=\bigl\{T^{1}_{k}~|~k\in\mathcal{T}\bigr\} \in L_{\mathbb{F}}^{\infty}(\overline{\mathcal{T}}; \mathbb{R}^{n\times n})$
	~and $T^{2}:=\bigl\{T^{2}_{k}~|~k\in \mathcal{T}\bigr\}\in L_{\mathbb{F}}^{\infty}(\overline{\mathcal{T}}; \mathbb{R}^{n\times n})$ satisfy the terminal conditions
	\begin{equation}\label{eq:4.22}
		\begin{array}{lll}
			T^{1}_{N} = G_{N},~~~~T^{2}_{N} = H_{N},
		\end{array}
	\end{equation}
	and $\phi^{1}:=\bigl\{\phi^{1}_{k}~|~k\in\mathcal{T}\bigr\}\in L_{\mathbb{F}}^2(\overline{\mathcal{T}}; \mathbb{R}^{n})$ and $\phi^{2}:=\bigl\{\phi^{2}_{k}~|~k\in\mathcal{T}\bigr\}\in L_{\mathbb{F}}^2(\overline{\mathcal{T}}; \mathbb{R}^{n})$ satisfying following BS$\Delta$Es:
	\begin{equation}\label{eq:4.23}
		\left\{\begin{array}{lll}
			\phi^{1}_{k} = f^{1}_{k+1}+g^{1}_{k+1}\mathbb{E}[\phi^{1}_{k+1}\mid \mathcal {F}_{k-1}]+h^{1}_{k+1}\mathbb{E}[\phi^{1}_{k+1}\omega_{k}\mid \mathcal {F}_{k-1}]
			\\
			\phi^{1}_{N} = g_{N},
		\end{array}
		\right.
	\end{equation}
	and
	\begin{equation}\label{eq:4.24}
		\left\{\begin{array}{lll}
			\phi^{2}_{k} = f^{2}_{k+1}+g^{2}_{k+1}\mathbb{E}[\phi^{2}_{k+1}\mid \mathcal {F}_{k-1}]+h^{2}_{k+1}\mathbb{E}[\phi^{2}_{k+1}\omega_{k}\mid \mathcal {F}_{k-1}]
			\\
			\phi^{2}_{N} = h_{N},
		\end{array}
		\right.
	\end{equation}
	\normalsize{where $f^{1}_{k+1},f^{2}_{k+1},g^{1}_{k+1},g^{2}_{k+1}$, $h^{1}_{k+1}$ and $h^{2}_{k+1}$ are 
 stochastic processes.~Furthermore, using the notations introduced in Appendix A, the Hamiltonian systems \eqref{eq:4.19} and \eqref{eq:4.20} can be equivalently rewritten as
 }

\begin{equation}\label{eq:4.25}
	\left\{
	\begin{aligned}
		x_{k+1} &= {A}_{k}x_{k} + \Lambda^{2}_{k}\pi_{k} + b_{k} + \left({D}_{k}x_{k} + \Lambda^{4}_{k}\pi_{k} + \sigma_{k}\right)\omega_{k}, \\
		Y_{k} &= \tilde{A}_{k}^{\top}\mathbb{E}\big[Y_{k+1} \mid \mathcal{F}_{k-1}\big] + \tilde{D}_{k}^{\top}\mathbb{E}\big[Y_{k+1}\omega_{k} \mid \mathcal{F}_{k-1}\big] \\
		&\quad + \Lambda^{5}_{k}x_{k} + \Lambda_{k}^{6\top}\pi_{k} + \lambda^{1}_{k}, \\
		x_{0} &= \xi \in \mathbb{R}^n,
		Y_{N} = Gx_{N} + g, \\
		0 &= \Lambda^{1\top}_{k}\mathbb{E}\big[Y_{k+1} \mid \mathcal{F}_{k-1}\big] + \Lambda^{3\top}_{k}\mathbb{E}\big[Y_{k+1}\omega_{k} \mid \mathcal{F}_{k-1}\big] \\
		&\quad + \Lambda^{6}_{k}\boldsymbol{I}_{n}x_{k} + \Lambda^{7}_{k}\pi_{k} + \lambda^{2}_{k}, \quad k \in \mathcal{T}.
	\end{aligned}
	\right.
\end{equation}

\begin{rmk}\label{rmk:4.30}
	The unified Hamiltonian system \eqref{eq:4.25} is an equivalent reformulation of the stationarity conditions in Theorem~\ref{thm:4.1}. It integrates the dynamics, adjoint equations, and equilibrium conditions of both players into a single compact form, where the coefficients \( \Lambda_k^i \) and \( \lambda_k^j ,i=1\cdots7,~j=1,2\) are appropriately defined to recover the original conditions \eqref{eq:4.7} and \eqref{eq:4.8}. This unified system thus characterizes the same open-loop Nash equilibrium.
\end{rmk}
	In view of \eqref{eq:4.21}, the relationship between the state process \( x_k \), associated with the open-loop Nash equilibrium point \( \pi_k \), and the corresponding adjoint process \( Y_k \) can be expressed as follows:
	
	\begin{equation}\label{eq:4.26}
		Y_k = T_k x_k + \phi_k,
	\end{equation}
	where \( T := \{ T_k ~|~ k \in \mathcal{T} \} \in L_{\mathbb{F}}^{\infty}(\overline{\mathcal{T}}; \mathbb{R}^{2n \times n}) \) satisfies the condition
	\begin{equation}\label{eq:4.27}
		T_N = G.
	\end{equation}
	Additionally, \( \phi := \{ \phi_k ~|~ k \in \mathcal{T} \} \in L_{\mathbb{F}}^2(\overline{\mathcal{T}}; \mathbb{R}^{2n}) \) satisfies the following BS$\Delta$E:
	\begin{equation}\label{eq:4.28}
		\left\{
		\begin{array}{lll}
			\phi_k = f_{k+1} + g_{k+1} \mathbb{E}[\phi_{k+1} \mid \mathcal{F}_{k-1}] + h_{k+1} \mathbb{E}[\phi_{k+1} \omega_k \mid \mathcal{F}_{k-1}], \\
			\phi_N = g.
		\end{array}
		\right.
	\end{equation}
	Considering the adjoint equations \( Y = \{Y_{k}\}_{k=0}^{N} \in L_{\mathbb{F}}^2(\overline{\mathcal{T}}; \mathbb{R}^{2n}) \) in \eqref{eq:4.25} , the relationship \eqref{eq:4.26}, the state equation \eqref{eq:2.1}, and using the notations in Appendix B, we deduce that
	
	\begin{eqnarray}\label{eq:4.30}
			\begin{split}
				&T_{k}x_{k}+\phi_{k}
				\\=&Y_{k}
				\\
				=&\tilde{A}_{k}^{\top}\mathbb{E}[Y_{k+1} \mid \mathcal{F}_{k-1}] + \tilde{D}_{k}^{\top}\mathbb{E}[Y_{k+1}\omega_{k} \mid \mathcal{F}_{k-1}] +\Lambda^{5}_{k}x_{k} \\&+ \Lambda_{k}^{6 \top}\pi_{k}+\lambda^{1}_{k}
				\\
				=&\tilde{A}_{k}^{\top}\mathbb{E}[	T_{k+1}x_{k+1}+\phi_{k+1} \mid \mathcal{F}_{k-1}] + \tilde{D}_{k}^{\top}\mathbb{E}[\big(	T_{k+1}x_{k+1}+\phi_{k+1}\big)\\&\times\omega_{k} \mid \mathcal{F}_{k-1}] +\Lambda^{5}_{k}x_{k} + \Lambda_{k}^{6 \top}\pi_{k}+\lambda^{1}_{k}
				\\=&\tilde{A}_{k}^{\top}\mathbb{E}[	T_{k+1}\big(A_{k}x_{k}+\Lambda^{2}_{k}\pi_{k}+b_{k}+\left(D_{k}x_{k}+\Lambda^{4}_{k}\pi_{k}+\sigma_{k}\right)\\&\times\omega_{k}\big)+\phi_{k+1} \mid \mathcal{F}_{k-1}] + \tilde{D}_{k}^{\top}\mathbb{E}[\big(	T_{k+1}\big(A_{k}x_{k}+\Lambda^{2}_{k}\pi_{k}+b_{k}\\&+\left(D_{k}x_{k}+\Lambda^{4}_{k}\pi_{k}+\sigma_{k}\right)\omega_{k}\big)+\phi_{k+1}\big)\omega_{k} \mid \mathcal{F}_{k-1}] +\Lambda^{5}_{k}x_{k} \\&+ \Lambda_{k}^{6 \top}\pi_{k}+\lambda^{1}_{k}
				\\=&\Delta\left(T_{k+1}\right)x_{k}+\mathscr{L}^{\top}\left(T_{k+1}\right)\pi_{k}+\Theta\left(T_{k+1},\phi_{k+1}\right).
			\end{split}
	\end{eqnarray}
	$\!\!\!\!\!$~To elaborate further, using the notations in Appendix B, we substitute the relationship in \eqref{eq:4.26} into the stationarity condition in \eqref{eq:4.25} for the open-loop Nash-Equilibrium
	point $\pi_k$ which yields
	{\begin{equation} \label{eq:4.31}
			\begin{split}
				&0\\=&\Lambda^{1\top}_{k}\mathbb{E}[Y_{k+1} \mid \mathcal{F}_{k-1}] + \Lambda^{3\top}_{k}\mathbb{E}[Y_{ k+1}\omega_{k} \mid \mathcal{F}_{k-1}] +  \Lambda^{6}_{k}\boldsymbol{I}_{n}x_{k} \\&+ \Lambda^{7}_{k}\pi_{k}+\lambda^{2}_{k}
				\\
				=&\Lambda^{1\top}_{k}\mathbb{E}[T_{k+1}x_{k+1}+\phi_{k+1}\mid \mathcal{F}_{k-1}] + \Lambda^{3\top}_{k}\mathbb{E}[\big(T_{k+1}x_{k+1}+\phi_{k+1}\big)\\&\times\omega_{k} \mid \mathcal{F}_{k-1}] + \Lambda^{6}_{k}\boldsymbol{I}_{n}x_{k} + \Lambda^{7}_{k}\pi_{k}+\lambda^{2}_{k}
				\\
				=&\Lambda^{1\top}_{k}\mathbb{E}[T_{k+1}\big(	A_{k}x_{k}+\Lambda^{2}_{k}\pi_{k}+b_{k}+\left(D_{k}x_{k}+\Lambda^{4}_{k}\pi_{k}+\sigma_{k}\right)\omega_{k} \big)\\&+\phi_{k+1}\mid \mathcal{F}_{k-1}] + \Lambda^{3\top}_{k}\mathbb{E}[\big(T_{k+1}\big(	A_{k}x_{k}+\Lambda^{2}_{k}\pi_{k}+b_{k}+\big(D_{k}x_{k}\\&+\Lambda^{4}_{k}\pi_{k}+\sigma_{k}\big)\omega_{k} \big)+\phi_{k+1}\big)\omega_{k} \mid \mathcal{F}_{k-1}] + \Lambda^{6}_{k}\boldsymbol{I}_{n}x_{k} + \Lambda^{7}_{k}\pi_{k}+\lambda^{2}_{k}
				\\
				=&\Gamma\left(T_{k+1}\right)x_{k}+\Upsilon\left(T_{k+1}\right)\pi_{k}+\Phi\left(T_{k+1},\phi_{k+1}\right).
			\end{split}
	\end{equation}}
	\begin{rmk}\label{rmk:4.3}
The Riccati recursion leads to a control representation involving the matrix
\begin{equation}\nonumber
	\begin{split}
		\Upsilon(T&_{k+1}) \\:=& \Lambda^{7}_{k}+\Lambda^{1\top}_{k}\mathbb{E}[T_{k+1} \mid \mathcal{F}_{k-1}]\Lambda^{2}_{k}+\Lambda^{1\top}_{k}\mathbb{E}[T_{k+1}\omega_{k} \mid \mathcal{F}_{k-1}]\Lambda^{4}_{k}\\&+\Lambda^{3\top}_{k}\mathbb{E}[T_{k+1}\omega_{k} \mid \mathcal{F}_{k-1}]\Lambda^{2}_{k}+\Lambda^{3\top}_{k}\mathbb{E}[T_{k+1}\omega^{2}_{k} \mid \mathcal{F}_{k-1}]\Lambda^{4}_{k},
	\end{split}
\end{equation}
whose bounded invertibility is required for computing the optimal control. In the previous, Assumption 2.3 guarantees the uniform convexity of each player's performance index and the strongly regular solvability of the Riccati equation in the single-player optimal control setting. We now provide sufficient conditions ensuring that \(\Upsilon(T_{k+1})\) is indeed invertible.~We impose the following assumptions:

\begin{itemize}
	\item[(i)] \(\Lambda^7_k \succ 0\) for all \(k\), i.e., the control cost matrix is uniformly positive definite for all k $\in\mathcal{T}$;
	\item[(ii)] \(T_N \succeq 0\) and \(T_k\) are generated recursively in a way that preserves positive semi-definiteness;
	\item[(iii)] The matrices \(\Lambda^1_k, \Lambda^2_k, \Lambda^3_k, \Lambda^4_k\) are uniformly bounded and satisfy structural non-degeneracy.
\end{itemize}
Under these assumptions, we can show via backward induction that the Riccati matrices \(T_k\) remain positive semi-definite for all \(k\in\mathcal{T}\). Starting from \(T_N = G \succeq 0\), and noting that the recursive formula maintains non-negativity due to the structure of the difference Riccati equation, we conclude that \(T_{k+1} \succeq 0\) implies $\Upsilon(T_{k+1}) \succ 0$.~Hence, \(\Upsilon(T_{k+1})\) is positive definite and invertible for all \(k\in\mathcal{T}\), which ensures the well-posedness of the feedback representation of the open-loop Nash equilibrium.
\end{rmk}

Under the assumptions in the Remark 4.3, the closed-loop representation of open-loop Nash-Equilibrium point $\pi_k$ can be expressed in terms of state feedback as follows:
	\begin{eqnarray} \label{eq:4.32}
		\begin{array}{lll}
			\pi_{k}=-\Upsilon\left(T_{k+1}\right)^{-1}\bigg[\Gamma\left(T_{k+1}\right)x_{k}+\Phi\left(T_{k+1},\phi_{k+1}\right)\bigg], k\in \mathcal{T}.
		\end{array}
	\end{eqnarray}
	Next,  substituting $\pi_{k}$ in \eqref{eq:4.32} into  \eqref{eq:4.30}, we get that
	{\begin{equation}\label{eq:4.33}
			\begin{array}{lll}
				\begin{split}
					&T_{k}x_{k}+\phi_{k}
					\\=&Y_{k}
					\\
					=&\Delta\left(T_{k+1}\right)x_{k}+\mathscr{L}^{\top}\left(T_{k+1}\right)\pi_{k}+\Theta\left(T_{k+1},\phi_{k+1}\right)
					\\=&\Delta\left(T_{k+1}\right)x_{k}-\mathscr{L}^{\top}\left(T_{k+1}\right)\Upsilon\left(T_{k+1}\right)^{-1}\bigg[\Gamma\left(T_{k+1}\right)x_{k}\\&+\Phi\left(T_{k+1},\phi_{k+1}\right)\bigg]+\Theta\left(T_{k+1},\phi_{k+1}\right)
					\\=&\bigg[\Delta\left(T_{k+1}\right)-\mathscr{L}^{\top}\left(T_{k+1}\right)\Upsilon\left(T_{k+1}\right)^{-1}\Gamma\left(T_{k+1}\right)\bigg]x_{k}\\&+\Theta\left(T_{k+1},\phi_{k+1}\right)-\mathscr{L}^{\top}\left(T_{k+1}\right)\Upsilon\left(T_{k+1}\right)^{-1}\Phi\left(T_{k+1},\phi_{k+1}\right)
				\end{split}
			\end{array}
	\end{equation}}
	$\!\!$By comparing both sides of \eqref{eq:4.33}, we can derive the expressions for $T$, which is the solution to the following difference Riccati equations:
	\small{
	\begin{equation}\label{eq:4.34}
			\left\{\begin{array}{lll}
				\begin{split}
					T_{k}=&\Delta\left(T_{k+1}\right)-\mathscr{L}^{\top}\left(T_{k+1}\right)\Upsilon\left(T_{k+1}\right)^{-1}\Gamma\left(T_{k+1}\right)
					\\
					=&\bigg[\Lambda^{5}_{k}+\tilde{A}_{k}^{\top}\mathbb{E}[T_{k+1} \mid \mathcal{F}_{k-1}]A_{k}+\tilde{A}_{k}^{\top}\mathbb{E}[T_{k+1}\omega_{k} \mid \mathcal{F}_{k-1}]D_{k}\\&+\tilde{D}_{k}^{\top}\mathbb{E}[T_{k+1}\omega_{k} \mid \mathcal{F}_{k-1}]A_{k}+\tilde{D}_{k}^{\top}\mathbb{E}[T_{k+1}\omega_{k}^{2} \mid \mathcal{F}_{k-1}]D_{k}\bigg]
					\\&-\bigg[\Lambda^{6}_{k}+\Lambda^{2\top}_{k}\mathbb{E}[T_{k+1} \mid \mathcal{F}_{k-1}]\tilde{A}_{k}+\Lambda^{4\top}_{k}\mathbb{E}[T_{k+1}\omega_{k} \mid \mathcal{F}_{k-1}]\tilde{A}_{k}\\&+\Lambda^{2\top}_{k}\mathbb{E}[T_{k+1}\omega_{k} \mid \mathcal{F}_{k-1}]\tilde{D}_{k}+\Lambda^{4\top}_{k}\mathbb{E}[T_{k+1}\omega_{k}^{2} \mid \mathcal{F}_{k-1}]\tilde{D}_{k}\bigg]^{\top}\\&\times\bigg[\Lambda^{7}_{k}+\Lambda^{1\top}_{k}\mathbb{E}[T_{k+1} \mid \mathcal{F}_{k-1}]\Lambda^{2}_{k}+\Lambda^{1\top}_{k}\mathbb{E}[T_{k+1}\omega_{k} \mid \mathcal{F}_{k-1}]\Lambda^{4}_{k}\\&+\Lambda^{3\top}_{k}\mathbb{E}[T_{k+1}\omega_{k} \mid \mathcal{F}_{k-1}]\Lambda^{2}_{k}+\Lambda^{3\top}_{k}\mathbb{E}[T_{k+1}\omega^{2}_{k} \mid \mathcal{F}_{k-1}]\Lambda^{4}_{k}\bigg]^{-1}\\&\times\bigg[\Lambda^{6}_{k}\boldsymbol{I}_{n}+\Lambda^{1\top}_{k}\mathbb{E}[T_{k+1} \mid \mathcal{F}_{k-1}]A_{k}+\Lambda^{3\top}_{k}\mathbb{E}[T_{k+1}\omega_{k} \mid \mathcal{F}_{k-1}]\\&\times A_{k}+\Lambda^{1\top}_{k}\mathbb{E}[T_{k+1}\omega_{k} \mid \mathcal{F}_{k-1}]D_{k}+\Lambda^{3\top}_{k}\mathbb{E}[T_{k+1}\omega^{2}_{k} \mid \mathcal{F}_{k-1}]D_{k}\bigg],
					\\
					T_{N}=&G.
				\end{split}
			\end{array}
			\right.
	\end{equation}}
	For any $k\in\mathcal{T}$, $\phi$ satisfies the following BS$\Delta$E:
		\begin{equation}\nonumber
		\left\{\begin{array}{lll}
			\begin{split}
				\phi_{k}=&\Theta\left(T_{k+1},\phi_{k+1}\right)-\mathscr{L}^{\top}\left(T_{k+1}\right)\Upsilon\left(T_{k+1}\right)^{-1}\Phi\left(T_{k+1},\phi_{k+1}\right)
				\\=&-\biggl\{\bigg[\Lambda^{6}_{k}+\Lambda^{2\top}_{k}\mathbb{E}[T_{k+1} \mid \mathcal{F}_{k-1}]\tilde{A}_{k}+\Lambda^{4\top}_{k}\mathbb{E}[T_{k+1}\omega_{k} \mid \mathcal{F}_{k-1}]\tilde{A}_{k}\\&+\Lambda^{2\top}_{k}\mathbb{E}[T_{k+1}\omega_{k} \mid \mathcal{F}_{k-1}]\tilde{D}_{k}+\Lambda^{4\top}_{k}\mathbb{E}[T_{k+1}\omega_{k}^{2} \mid \mathcal{F}_{k-1}]\tilde{D}_{k}\bigg]^{\top}\\&\times\bigg[\Lambda^{7}_{k}+\Lambda^{1\top}_{k}\mathbb{E}[T_{k+1} \mid \mathcal{F}_{k-1}]\Lambda^{2}_{k}+\Lambda^{1\top}_{k}\mathbb{E}[T_{k+1}\omega_{k} \mid \mathcal{F}_{k-1}]\Lambda^{4}_{k}\\&+\Lambda^{3\top}_{k}\mathbb{E}[T_{k+1}\omega_{k} \mid \mathcal{F}_{k-1}]\Lambda^{2}_{k}+\Lambda^{3\top}_{k}\mathbb{E}[T_{k+1}\omega^{2}_{k} \mid \mathcal{F}_{k-1}]\Lambda^{4}_{k}\bigg]^{-1}\\&\times\bigg[\lambda^{2}_{k}+\Lambda^{1}_{k}\big(\mathbb{E}[T_{k+1}\mid \mathcal{F}_{k-1}]b_{k}+\mathbb{E}[T_{k+1}\omega_{k}\mid \mathcal{F}_{k-1}]\sigma_{k}\big)\\&+\Lambda^{3}_{k}\big(\mathbb{E}[T_{k+1}\omega_{k}\mid \mathcal{F}_{k-1}]b_{k}+\mathbb{E}[T_{k+1}\omega_{k}^{2}\mid \mathcal{F}_{k-1}]\sigma_{k}\big)\bigg]\end{split}
			\end{array}
		\right.
	\end{equation}
\begin{equation}\label{eq:4.35}
\left\{\begin{array}{lll}
\begin{split}&-\bigg[\lambda^{1}_{k}+\tilde{A}_{k}\big(\mathbb{E}[T_{k+1}\mid \mathcal{F}_{k-1}]b_{k}+\mathbb{E}[T_{k+1}\omega_{k}\mid \mathcal{F}_{k-1}]\sigma_{k}\big)\\&+\tilde{D}_{k}\big(\mathbb{E}[T_{k+1}\omega_{k}\mid \mathcal{F}_{k-1}]b_{k}+\mathbb{E}[T_{k+1}\omega_{k}^{2}\mid \mathcal{F}_{k-1}]\sigma_{k}\big)\bigg]\biggr\}
				+\biggl\{\tilde{A}_{k}\\&-\bigg[\Lambda^{6}_{k}+\Lambda^{2\top}_{k}\mathbb{E}[T_{k+1} \mid \mathcal{F}_{k-1}]\tilde{A}_{k}+\Lambda^{4\top}_{k}\mathbb{E}[T_{k+1}\omega_{k} \mid \mathcal{F}_{k-1}]\tilde{A}_{k}\\&+\Lambda^{2\top}_{k}\mathbb{E}[T_{k+1}\omega_{k} \mid \mathcal{F}_{k-1}]\tilde{D}_{k}+\Lambda^{4\top}_{k}\mathbb{E}[T_{k+1}\omega_{k}^{2} \mid \mathcal{F}_{k-1}]\tilde{D}_{k}\bigg]^{\top}\\&\times\bigg[\Lambda^{7}_{k}+\Lambda^{1\top}_{k}\mathbb{E}[T_{k+1} \mid \mathcal{F}_{k-1}]\Lambda^{2}_{k}+\Lambda^{1\top}_{k}\mathbb{E}[T_{k+1}\omega_{k} \mid \mathcal{F}_{k-1}]\Lambda^{4}_{k}\\&+\Lambda^{3\top}_{k}\mathbb{E}[T_{k+1}\omega_{k} \mid \mathcal{F}_{k-1}]\Lambda^{2}_{k}+\Lambda^{3\top}_{k}\mathbb{E}[T_{k+1}\omega^{2}_{k} \mid \mathcal{F}_{k-1}]\Lambda^{4}_{k}\bigg]^{-1}\\&\times\Lambda^{1}_{k}
				\biggr\} \mathbb{E}[\phi^{1}_{k+1}\mid \mathcal {F}_{k-1}]
				+\biggl\{\tilde{D}_{k}-\bigg[\Lambda^{6}_{k}+\Lambda^{2\top}_{k}\mathbb{E}[T_{k+1} \mid \mathcal{F}_{k-1}]
			\\
				    &\times\tilde{A}_{k}+\Lambda^{4\top}_{k}\mathbb{E}[T_{k+1}\omega_{k} \mid \mathcal{F}_{k-1}]\tilde{A}_{k}+\Lambda^{2\top}_{k}\mathbb{E}[T_{k+1}\omega_{k} \mid \mathcal{F}_{k-1}]\tilde{D}_{k}\\&+\Lambda^{4\top}_{k}\mathbb{E}[T_{k+1}\omega_{k}^{2} \mid \mathcal{F}_{k-1}]\tilde{D}_{k}\bigg]^{\top}\times\bigg[\Lambda^{7}_{k}+\Lambda^{1\top}_{k}\mathbb{E}[T_{k+1} \mid \mathcal{F}_{k-1}]\\&\times\Lambda^{2}_{k}+\Lambda^{1\top}_{k}\mathbb{E}[T_{k+1}\omega_{k} \mid \mathcal{F}_{k-1}]\Lambda^{4}_{k}+\Lambda^{3\top}_{k}\mathbb{E}[T_{k+1}\omega_{k} \mid \mathcal{F}_{k-1}]\Lambda^{2}_{k}\\&+\Lambda^{3\top}_{k}\mathbb{E}[T_{k+1}\omega^{2}_{k} \mid \mathcal{F}_{k-1}]\Lambda^{4}_{k}\bigg]^{-1} \Lambda^{3}_{k}
					\biggr\}\times\mathbb{E}[\phi^{1}_{k+1}\omega_{k}\mid \mathcal {F}_{k-1}],
					\\
					\phi_{N}=&g.
				\end{split}
			\end{array}
			\right.
	\end{equation}
	Then using \eqref {eq:4.35}, we get that   $f_{k+1},g_{k+1}$ and $h_{k+1}$ in \eqref{eq:4.28} should satisfy the following expressions:
	
	{\small\begin{equation}\nonumber
				\begin{array}{lll}
					\begin{split}
						f_{k+1}&
						\\=&\bigg[\lambda^{1}_{k}+\tilde{A}_{k}\big(\mathbb{E}[T_{k+1}\mid \mathcal{F}_{k-1}]b_{k}+\mathbb{E}[T_{k+1}\omega_{k}\mid \mathcal{F}_{k-1}]\sigma_{k}\big)\\&+\tilde{D}_{k}\big(\mathbb{E}[T_{k+1}\omega_{k}\mid \mathcal{F}_{k-1}]b_{k}+\mathbb{E}[T_{k+1}\omega_{k}^{2}\mid \mathcal{F}_{k-1}]\sigma_{k}\big)\bigg]\\&-\bigg[\Lambda^{6}_{k}+\Lambda^{2\top}_{k}\mathbb{E}[T_{k+1} \mid \mathcal{F}_{k-1}]\tilde{A}_{k}+\Lambda^{4\top}_{k}\mathbb{E}[T_{k+1}\omega_{k} \mid \mathcal{F}_{k-1}]\\&\times\tilde{A}_{k}+\Lambda^{2\top}_{k}\mathbb{E}[T_{k+1}\omega_{k} \mid \mathcal{F}_{k-1}]\tilde{D}_{k}+\Lambda^{4\top}_{k}\mathbb{E}[T_{k+1}\omega_{k}^{2} \mid \mathcal{F}_{k-1}]\\&\times\tilde{D}_{k}\bigg]^{\top}\bigg[\Lambda^{7}_{k}+\Lambda^{1\top}_{k}\mathbb{E}[T_{k+1} \mid \mathcal{F}_{k-1}]\Lambda^{2}_{k}+\Lambda^{1\top}_{k}\mathbb{E}[T_{k+1}\omega_{k} \mid \mathcal{F}_{k-1}]\\&\times\Lambda^{4}_{k}+\Lambda^{3\top}_{k}\mathbb{E}[T_{k+1}\omega_{k} \mid \mathcal{F}_{k-1}]\Lambda^{2}_{k}+\Lambda^{3\top}_{k}\mathbb{E}[T_{k+1}\omega^{2}_{k} \mid \mathcal{F}_{k-1}]\Lambda^{4}_{k}\bigg]^{-1}\\&\times\bigg[\lambda^{2}_{k}+\Lambda^{1}_{k}\big(\mathbb{E}[T_{k+1}\mid \mathcal{F}_{k-1}]b_{k}+\mathbb{E}[T_{k+1}\omega_{k}\mid \mathcal{F}_{k-1}]\sigma_{k}\big)+\Lambda^{3}_{k}\\&\times\big(\mathbb{E}[T_{k+1}\omega_{k}\mid \mathcal{F}_{k-1}]b_{k}+\mathbb{E}[T_{k+1}\omega_{k}^{2}\mid \mathcal{F}_{k-1}]\sigma_{k}\big)\bigg],
					\end{split}
			\end{array}
		\end{equation}}
{\small\begin{equation}\nonumber
\begin{array}{lll}
\begin{split}
						g_{k+1}&\\=&\tilde{A}_{k}-\bigg[\Lambda^{6}_{k}+\Lambda^{2\top}_{k}\mathbb{E}[T_{k+1} \mid \mathcal{F}_{k-1}]\tilde{A}_{k}+\Lambda^{4\top}_{k}\mathbb{E}[T_{k+1}\omega_{k} \mid \mathcal{F}_{k-1}]\\&\times\tilde{A}_{k}+\Lambda^{2\top}_{k}\mathbb{E}[T_{k+1}\omega_{k} \mid \mathcal{F}_{k-1}]\tilde{D}_{k}+\Lambda^{4\top}_{k}\mathbb{E}[T_{k+1}\omega_{k}^{2} \mid \mathcal{F}_{k-1}]\\&\times\tilde{D}_{k}\bigg]^{\top}\bigg[\Lambda^{7}_{k}+\Lambda^{1\top}_{k}\mathbb{E}[T_{k+1} \mid \mathcal{F}_{k-1}]\Lambda^{2}_{k}+\Lambda^{1\top}_{k}\mathbb{E}[T_{k+1}\omega_{k} \mid \mathcal{F}_{k-1}]\\&\times\Lambda^{4}_{k}+\Lambda^{3\top}_{k}\mathbb{E}[T_{k+1}\omega_{k} \mid \mathcal{F}_{k-1}]\Lambda^{2}_{k}+\Lambda^{3\top}_{k}\mathbb{E}[T_{k+1}\omega^{2}_{k} \mid \mathcal{F}_{k-1}]\Lambda^{4}_{k}\bigg]^{-1}\Lambda^{1}_{k}
						,
				\end{split}
			\end{array}
		\end{equation}}
	{\small\begin{equation}\nonumber
\begin{array}{lll}
\begin{split}
						h_{k+1}&\\=&\tilde{D}_{k}-\bigg[\Lambda^{6}_{k}+\Lambda^{2\top}_{k}\mathbb{E}[T_{k+1} \mid \mathcal{F}_{k-1}]\tilde{A}_{k}+\Lambda^{4\top}_{k}\mathbb{E}[T_{k+1}\omega_{k} \mid \mathcal{F}_{k-1}]\\&\times\tilde{A}_{k}+\Lambda^{2\top}_{k}\mathbb{E}[T_{k+1}\omega_{k} \mid \mathcal{F}_{k-1}]\tilde{D}_{k}+\Lambda^{4\top}_{k}\mathbb{E}[T_{k+1}\omega_{k}^{2} \mid \mathcal{F}_{k-1}]\tilde{D}_{k}\bigg]^{\top}\\&\times\bigg[\Lambda^{7}_{k}+\Lambda^{1\top}_{k}\mathbb{E}[T_{k+1} \mid \mathcal{F}_{k-1}]\Lambda^{2}_{k}+\Lambda^{1\top}_{k}\mathbb{E}[T_{k+1}\omega_{k} \mid \mathcal{F}_{k-1}]\Lambda^{4}_{k}\\&+\Lambda^{3\top}_{k}\mathbb{E}[T_{k+1}\omega_{k} \mid \mathcal{F}_{k-1}]\Lambda^{2}_{k}+\Lambda^{3\top}_{k}\mathbb{E}[T_{k+1}\omega^{2}_{k} \mid \mathcal{F}_{k-1}]\Lambda^{4}_{k}\bigg]^{-1} \Lambda^{3}_{k}.
					\end{split}
				\end{array}
		\end{equation}}
		$\!\!$Given the results above, we now present the state feedback representation of the open-loop Nash-Equilibrium point $\pi$ for the discrete-time stochastic non-zero-sum difference games.
		\begin{thm}\label{thm:4.4}
			Let Assumptions \ref{ass:2.1}-\ref{ass:2.3} hold.~Suppose that the Riccati Eqs.~\eqref{eq:4.34} admit
			a set of solutions $T \in  L_{\mathbb{F}}^{\infty}(\overline{\mathcal{T}}; \mathbb{R}^{2n \times n}) $ and the BS$\Delta$E \eqref{eq:4.35} admit
			a set of solutions $\phi\in  L_{\mathbb{F}}^2(\overline{\mathcal{T}}; \mathbb{R}^{2n})$.~Then Problem \ref{pro:2.2} is  open-loop solvable, and the corresponding open-loop Nash-Equilibrium point $\pi=\big(\mathbf{u},\mathbf{v}\big)^{\top}\in   L_{\mathbb{F}}^2({\mathcal{T}}; \mathbb{R}^{m+l})$ has the following state feedback representation:
			\begin{eqnarray} \label{eq:4.36}
				\begin{split}
					&\pi_{k}=-\Upsilon\left(T_{k+1}\right)^{-1}\bigg[\Gamma\left(T_{k+1}\right)x_{k}+\Phi\left(T_{k+1},\phi_{k+1}\right)\bigg],
				\end{split}
			\end{eqnarray}
			where $x \in L_{\mathbb{F}}^2(\mathcal{T}; \mathbb{R}^{2n})$ is the solution of the following S$\Delta$E:
			\begin{equation}\label{eq:4.37}
				\left\{
				\begin{array}{ll}
					x^{\pi}_{k+1}={A}_{k}x^{\pi}_{k}+\Lambda^{2}_{k}{\pi}+{b}_{k}+\left({D}_{k}x^{\pi}_{k}+\Lambda^{4}_{k}{\pi}+{\sigma}_{k}\right)\omega_{k}, \\
					x^{\pi}_{0} = \xi \in \mathbb{R}^{n}, \quad k \in \mathcal{T}.
				\end{array}
				\right.
			\end{equation}
		\end{thm}
		\begin{proof}
			If the Riccati Eqs. \eqref{eq:4.34} admit
			a set of solutions $T$ and the BS$\Delta$E \eqref{eq:4.35} admit
			a set of solutions $\phi$.  Using the notations in Appendix B,  define an admissible control pair $\pi=\big(\mathbf{u},\mathbf{v}\big)^{\top}$ which has the following state feedback representation:
			\begin{eqnarray} \label{eq:4.38}
				\begin{split}
					&&\pi_{k}=-\Upsilon\left(T_{k+1}\right)^{-1}\bigg[\Gamma\left(T_{k+1}\right)x_{k}+\Phi\left(T_{k+1},\phi_{k+1}\right)\bigg],
				\end{split}
			\end{eqnarray}
			where $x$ is the solution of the following S$\Delta$E :
			\begin{equation}\label{eq:4.39}
				\left\{
				\begin{array}{ll}
					x^{\pi}_{k+1}={A}_{k}x^{\pi}_{k}+\Lambda^{2}_{k}{\pi_{k}}+{b}_{k}+\left({D}_{k}x^{\pi}_{k}+\Lambda^{4}_{k}{\pi_{k}}+{\sigma}_{k}\right)\omega_{k}, \\
					x^{\pi}_{0} = \xi \in \mathbb{R}^{n}, \quad k \in \mathcal{T}.
				\end{array}
				\right.
			\end{equation}
			Define
			\begin{equation}\label{eq:4.40}
				Y_{k}\equiv T_{k}x_{k}+\phi_{k}
				~~ k\in\mathcal{T},
			\end{equation}
			From the derivations of the  stochastic Riccati Eqs.~\eqref{eq:4.34}~and BS$\Delta$E~\eqref{eq:4.35} presented earlier, we can derive $Y_{k}$ satisfies the following BS$\Delta$E:
			\begin{equation}\label{eq:4.41}
				\left\{\begin{array}{lll}
					\begin{split}
						Y_{k} =& \tilde{A}_{k}^{\top}\mathbb{E}[Y_{k+1} \mid \mathcal{F}_{k-1}] + \tilde{D}_{k}^{\top}\mathbb{E}[Y_{k+1}\omega_{k} \mid \mathcal{F}_{k-1}] +\Lambda^{5}_{k}x_{k} \\&+ \Lambda_{k}^{6 \top}\pi_{k}+\lambda^{1}_{k},
						\\
						Y_{ N}=&Gx_{N}+g,~~~ k\in\mathcal{T},
					\end{split}
				\end{array}
				\right.
			\end{equation}
			It's not hard to spot that \eqref{eq:4.41} is the adjoint equation of  the state equation \eqref{eq:4.39} related to $\big(\pi,x\big)$.~By performing further computations, we  can obtain the following stationarity conditions:
			\begin{equation}\label{eq:4.42}
				\begin{split}
					&
					\Lambda^{1\top}_{k}\mathbb{E}[Y_{k+1} \mid \mathcal{F}_{k-1}] + \Lambda^{3\top}_{k}\mathbb{E}[Y_{ k+1}\omega_{k} \mid \mathcal{F}_{k-1}] + \Lambda^{6}_{k}x_{k} + \Lambda^{7}_{k}\pi_{k}+\lambda^{2}_{k}
					= 0.
				\end{split}
			\end{equation}
			According to Theorem \ref{thm:4.1}, we get that $\pi=\big(\mathbf{u},\mathbf{v}\big)^{\top}$ is the open-loop Nash-Equilibrium point  specified by \eqref{eq:4.36};~$(\pi,x,Y)$ is the unique solution to the Hamiltonian system \eqref{eq:4.25}.  The proof is complete.
		\end{proof}
		%

In the previous theorem, we establish a sufficient condition for the open-loop solvability of Problem
\ref {pro:2.2}. Specifically, the invertibility of the Riccati equations plays a crucial role, but it can be challenging to verify. To address this issue, we next study a necessary and sufficient condition that does not require the verification of invertibility. For this purpose, we will adopt a Lyapunov-like approach to provide a more tractable criterion.

 we define the feedback gain matrix spaces for the two players as follows:
\[
\boldsymbol{\Theta}^1(\mathcal{T}) := L_{\mathbb{F}}^\infty(\mathcal{T}; \mathbb{R}^{m \times n}), \quad
\boldsymbol{\Theta}^2(\mathcal{T}) := L_{\mathbb{F}}^\infty(\mathcal{T}; \mathbb{R}^{l \times n}),
\]
and set the combined feedback space as
\[
\boldsymbol{\Theta}(\mathcal{T}) := \boldsymbol{\Theta}^1(\mathcal{T}) \times \boldsymbol{\Theta}^2(\mathcal{T}).
\]

We consider state-feedback strategies of the form
\[
\pi_k = \Pi_k x_k + \Sigma_k, \quad k \in \mathcal{T},
\]
where $\pi$ is compact form in Appendix A,
\[
\Pi_k = \begin{pmatrix} \Pi_k^1 \\ \Pi_k^2 \end{pmatrix}, \quad
\Sigma_k = \begin{pmatrix} \Sigma_k^1 \\ \Sigma_k^2 \end{pmatrix}
\]
and the overall strategy pair \((\Pi, \Sigma)\) satisfies
\[
(\Pi, \Sigma) = (\Pi^{1}, \Sigma^{1}; \Pi^{2}, \Sigma^{2}),\]
\[\text{with}\quad\Pi^1 \in \boldsymbol{\Theta}^{1}(\mathcal{T}),\quad \Pi^2 \in \boldsymbol{\Theta}^{2}(\mathcal{T}), \quad \Sigma^1 \in \mathcal{U}, \quad \Sigma^2 \in \mathcal{V}.
\]

Thus, we define the admissible closed-loop strategy space as
\[
\mathcal{S}(\mathcal{T}) := \boldsymbol{\Theta}^1(\mathcal{T}) \times \mathcal{U} \times \boldsymbol{\Theta}^2(\mathcal{T}) \times \mathcal{V}.
\]
Then the overall strategy
$
(\Pi, \Sigma) \in \mathcal{S}(\mathcal{T}).
$
and we may group the strategies by player and denote
\[
\Pi := \begin{pmatrix} \Pi^1 \\ \Pi^2 \end{pmatrix} \in \boldsymbol{\Theta}(\mathcal{T}), \quad \Sigma := \begin{pmatrix} \Sigma^1 \\ \Sigma^2 \end{pmatrix}\in\mathcal{U}\times\mathcal{V},
\]
so that the pair \((\Pi, \Sigma)\) represents a closed-loop strategy profile, which is assumed to be independent of the initial state \( \xi \).

\begin{defn}\label{defn1}
We say that the open-loop Nash equilibria of the discrete-time stochastic difference game (DTSDG) with initial time \(0\) admit a \emph{closed-loop representation}, if there exists a strategy tuple
\[
(\Pi, \Sigma) \in \mathcal{S}(\mathcal{T}),
\]
such that for any initial state \(\xi \in \mathbb{R}^n\) and all \(k \in \mathcal{T}\), the control law
\begin{equation}\label{eq:4000}
	\pi_k = \Pi_k x_k + \Sigma_k,
\end{equation}
constitutes an open-loop Nash equilibrium corresponding to the initial pair \((0, \xi)\), where the resulting state process \(x(\cdot) = x(\cdot; 0, \xi, \Pi, \Sigma)\) evolves according to the closed-loop system:
\begin{equation}\label{eq:440}
\left\{
\begin{aligned}
x_{k+1} =&\, \Big(A_k + \Lambda_k^2 \Pi_k + D_k \omega_k + \Lambda_k^4 \Pi_k \omega_k\Big) x_k \\
&\, + \Lambda_k^2 \Sigma_k + b_k + \Lambda_k^4 \Sigma_k \omega_k + \sigma_k \omega_k, \\
x_0 =&\, \xi \in \mathbb{R}^n, \quad k \in \mathcal{T}.
\end{aligned}
\right.
\end{equation}
\end{defn}

We now proceed to characterize the structure of such closed-loop representations of open-loop Nash equilibria.

\begin{thm}\label{thm:4.6}
Let Assumptions~\ref{ass:2.1}--\ref{ass:2.3} are satisfied and $(\Pi, \Sigma) \in \mathcal{S}(\mathcal{T})$. Then, the open-loop Nash equilibrium for Problem~\ref{pro:2.2} can be expressed in the closed-loop form~\eqref{eq:4000} if and only if the following both conditions hold:

\smallskip

\noindent\textbf{(i)} There exists a process \( T \in L_{\mathbb{F}}^{\infty}(\overline{\mathcal{T}}; \mathcal{S}_{+}^{n}) \)
solving the following backward Lyapunov-type stochastic difference equation:
\begin{equation}\label{eq:5.15}
\left\{
\begin{aligned}
T_k &= \Delta(T_{k+1}) + \mathscr{L}^{\top}(T_{k+1}) \Pi_{k+1}, \quad k \in \mathcal{T}, \\
T_N &= G,
\end{aligned}
\right.
\end{equation}
such that the stationarity condition
\begin{equation}\label{eq:5.17}
\Gamma(T_{k+1}) + \Upsilon(T_{k+1}) \Pi_{k+1} = 0
\end{equation}
holds almost surely and for almost every \( k \in \mathcal{T} \).

\smallskip

\noindent\textbf{(ii)} There exists a process \( \phi^1 \in L_{\mathbb{F}}^2(\mathcal{T}; \mathbb{R}^n) \) solving the backward stochastic difference equation:
\begin{equation}\label{eq:5.18}
\left\{
\begin{aligned}
\phi_k &= \Theta(T_{k+1}, \phi_{k+1}) + \mathscr{L}^{\top}(T_{k+1}) \Sigma_{k+1}, \quad k \in \mathcal{T}, \\
\phi_N &= g,
\end{aligned}
\right.
\end{equation}
such that the consistency condition
\begin{equation}\label{eq:5.19}
\Phi(T_{k+1}, \phi_{k+1}) + \Upsilon(T_{k+1}) \Sigma_{k+1} = 0
\end{equation}
holds almost surely and for almost every \( k \in \mathcal{T} \).
\end{thm}

\begin{proof}
	\textbf{Sufficiency:}
	\\
Let \eqref{eq:5.15}-\eqref{eq:5.19} hold. Fix an arbitrary initial state \( \xi \in \mathbb{R}^n \), and let \( x = x(\cdot; 0, \xi, \Pi, \Sigma) \) denote the state trajectory governed by the closed-loop system \eqref{eq:440}, under the control law \eqref{eq:4000}, with feedback gain pair \( (\Pi, \Sigma) \).  To demonstrate that the open-loop Nash equilibrium has a closed-form expression \( (\Pi, \Sigma) \), we refer to Remark 4.3. Specifically, it suffices to verify \( (x,\pi, Y) \) satisfy the Hamiltonian system \eqref{eq:4.25}.

Define
\[
Y_k = T_k x_k + \phi_k,~~Y_{N} = Gx_{N}+g
\]
where \( T_k \) and \( \phi_k \) are adapted processes of \eqref{eq:5.15} and \eqref{eq:5.17}, respectively.

\smallskip
\textit{Step 1: Verification of the adjoint equation.}

From the ansatz, direct computation yields:
\[
\begin{split}
	 &\tilde{A}_{k}^{\top} \mathbb{E}[Y_{k+1} \mid \mathcal{F}_{k-1}] + \tilde{D}_{k}^{\top} \mathbb{E}[Y_{k+1} \omega_{k} \mid \mathcal{F}_{k-1}] + \Lambda^{5}_{k} x_{k} + \Lambda_{k}^{6 \top} \pi_{k} + \lambda^{1}_{k} \\
	=& \tilde{A}_k^{\top} \mathbb{E}[Y_{k+1} \mid \mathcal{F}_{k-1}] + \tilde{D}_k^{\top} \mathbb{E}[Y_{k+1} \omega_k \mid \mathcal{F}_{k-1}] + \Lambda_k^5 x_k+ \Lambda_k^{6\top} \big( \Pi_{k} x_k  \\
	& + \Sigma_{k} \big) + \lambda_k^1 \\
	=& \tilde{A}_k^{\top} \mathbb{E}\bigg[T_{k+1} \biggl\{\Big(A_k + \Lambda_k^2 \Pi_k + D_k \omega_k + \Lambda_k^4 \Pi_k \omega_k\Big) x_k+ \Lambda_k^2 \Sigma_k + b_k \\&+ \Lambda_k^4 \Sigma_k \omega_k + \sigma_k \omega_k\biggr\} + \phi_{k+1} \mid \mathcal{F}_{k-1}\bigg] + \tilde{D}_k^{\top} \mathbb{E}\bigg[(T_{k+1} \biggl\{\Big(A_k + \Lambda_k^2 \Pi_k \\&+ D_k \omega_k + \Lambda_k^4 \Pi_k \omega_k\Big) x_k+ \Lambda_k^2 \Sigma_k + b_k + \Lambda_k^4 \Sigma_k \omega_k + \sigma_k \omega_k\biggr\} + \phi_{k+1}) \\
	&\times\omega_k \mid \mathcal{F}_{k-1}\bigg]  + \Lambda_k^5 x_k + \Lambda_k^{6\top} \left( \Pi_{k} x_k + \Sigma_{k} \right) + \lambda_k^1 \\
	=& \left( \Delta(T_{k+1}) + \mathscr{L}^{\top}(T_{k+1}) \Pi_{k+1} \right) x_k + \Theta(T_{k+1}, \phi_{k+1}) + \mathscr{L}^{\top}(T_{k+1}) \Sigma_{k+1}
   \\=&T_k x_k + \phi_k
   \\=&Y_k  .
\end{split}
\]
Therefore, under equations \eqref{eq:5.15} and \eqref{eq:5.18}, the process
$Y$ satisfies the backward equation of the Hamiltonian system \eqref{eq:4.25}.

\smallskip
\textit{Step 2: Verification of the stationarity condition.}
		\begin{align*}
			 & \Lambda^{1\top}_{k}\mathbb{E}[Y_{k+1} \mid \mathcal{F}_{k-1}] + \Lambda^{3\top}_{k}\mathbb{E}[Y_{ k+1}\omega_{k} \mid \mathcal{F}_{k-1}] + \Lambda^{6}_{k}\boldsymbol{I}_{n}x_{k} + \Lambda^{7}_{k}\pi_{k}+\lambda^{2}_{k}
			\\
			=& \Lambda_k^{1\top} \mathbb{E}[T_{k+1} x_{k+1} + \phi_{k+1} \mid \mathcal{F}_{k-1}] + \Lambda_k^{3\top} \mathbb{E}[(T_{k+1} x_{k+1} + \phi_{k+1}) \\
			&\times\omega_k \mid \mathcal{F}_{k-1}]  + \Lambda_k^6\boldsymbol{I}_{n} x_k + \Lambda_k^7 (\Pi_{k} x_k + \Sigma_{k}) + \lambda_k^2
			\\=& \Lambda_k^{1\top} \mathbb{E}\bigg[T_{k+1} \biggl\{\Big(A_k + \Lambda_k^2 \Pi_k + D_k \omega_k + \Lambda_k^4 \Pi_k \omega_k\Big) x_k+ \Lambda_k^2 \Sigma_k + b_k +\\& \Lambda_k^4 \Sigma_k \omega_k + \sigma_k \omega_k\biggr\} + \phi_{k+1} \mid \mathcal{F}_{k-1}\bigg] + \Lambda_k^{3\top} \mathbb{E}\bigg[(T_{k+1} \biggl\{\Big(A_k + \Lambda_k^2 \Pi_k \\&+ D_k \omega_k + \Lambda_k^4 \Pi_k \omega_k\Big) x_k+ \Lambda_k^2 \Sigma_k + b_k + \Lambda_k^4 \Sigma_k \omega_k + \sigma_k \omega_k\biggr\} + \phi_{k+1}) \\&\times\omega_k \mid \mathcal{F}_{k-1}\bigg] + \Lambda_k^6\boldsymbol{I}_{n} x_k + \Lambda_k^7 \big(\Pi_k x_k + \Sigma_k\big) + \lambda_k^2  \\
			=& \left[\Gamma(T_{k+1}) + \Upsilon(T_{k+1}) \Pi_{k+1} \right] x_k + \Phi(T_{k+1}, \phi_{k+1}) + \Upsilon(T_{k+1}) \Sigma_{k+1}
\\=&0.
		\end{align*}
	This reveals the stationarity condition in Hamiltonian system \eqref{eq:4.25} is satisfied,when \eqref{eq:5.17} and \eqref{eq:5.19} hold.~Therefore, the sufficiency is established.

	\textbf{Necessity:}
\\
Suppose the open-loop Nash equilibrium \((\Pi, \Sigma) \in \mathcal{S}(\mathcal{T})\) admits a closed-loop representation:
\begin{equation}\label{eq:4400}
    \pi_k = \Pi_k x_k + \Sigma_k,
\end{equation}
where the state process \(x(\cdot)\) satisfies the closed-loop system \eqref{eq:440}.~From the Hamiltonian system \eqref{eq:4.25}, there exists an adjoint process $\{Y_k\}_{k\in\mathcal{T}}$ satisfying~$Y_k = T_k x_k + \phi_k$.

Substituting the ansatz, closed-loop representation \eqref{eq:4400} and the closed-loop dynamics into \eqref{eq:4.25}'s adjoint equation yields:
\begin{equation}\label{eq:4.48}
\begin{split}
	Y_k =&\ \tilde{A}_{k}^{\top} \mathbb{E}[Y_{k+1} \mid \mathcal{F}_{k-1}] + \tilde{D}_{k}^{\top} \mathbb{E}[Y_{k+1} \omega_k \mid \mathcal{F}_{k-1}] + \Lambda^{5}_{k} x_k \\&+ \Lambda_{k}^{6 \top} \pi_k + \lambda^{1}_{k} \\
	=&\ \tilde{A}_k^{\top} \mathbb{E}[Y_{k+1} \mid \mathcal{F}_{k-1}] + \tilde{D}_k^{\top} \mathbb{E}[Y_{k+1} \omega_k \mid \mathcal{F}_{k-1}] + \Lambda_k^5 x_k \\&+ \Lambda_k^{6\top} \left( \Pi_{k} x_k + \Sigma_{k} \right) + \lambda_k^1 \\
	=&\ \tilde{A}_k^{\top} \mathbb{E}\bigg[T_{k+1} \biggl\{\Big(A_k + \Lambda_k^2 \Pi_k + D_k \omega_k + \Lambda_k^4 \Pi_k \omega_k\Big) x_k\\
	&+ \Lambda_k^2 \Sigma_k + b_k +  \Lambda_k^4 \Sigma_k \omega_k + \sigma_k \omega_k\biggr\} + \phi_{k+1} \mid \mathcal{F}_{k-1}\bigg]  \\
	&+ \tilde{D}_k^{\top} \mathbb{E}\bigg[(T_{k+1} \biggl\{\Big(A_k + \Lambda_k^2 \Pi_k + D_k \omega_k + \Lambda_k^4 \Pi_k \omega_k\Big) x_k \\
	&+ \Lambda_k^2 \Sigma_k + b_k + \Lambda_k^4 \Sigma_k \omega_k + \sigma_k \omega_k\biggr\} + \phi_{k+1}) \omega_k \mid \mathcal{F}_{k-1}\bigg]  \\
	&+ \Lambda_k^5 x_k + \Lambda_k^{6\top} \left( \Pi_{k} x_k + \Sigma_{k} \right) + \lambda_k^1 \\
	=&\ \left( \Delta(T_{k+1}) + \mathscr{L}^{\top}(T_{k+1}) \Pi_{k+1} \right) x_k + \Theta(T_{k+1}, \phi_{k+1}) \\
	&+ \mathscr{L}^{\top}(T_{k+1}) \Sigma_{k+1} \\
	=&\ T_k x_k + \phi_k.
\end{split}
\end{equation}

Here, equations \eqref{eq:5.15} and \eqref{eq:5.18} are used.

Furthermore,~substituting the ansatz, closed-loop representation \eqref{eq:4400} and the closed-loop dynamics into the stationarity condition in the Hamiltonian system \eqref{eq:4.25}, we obtain
\begin{align*}
	0=& \Lambda^{1\top}_{k}\mathbb{E}[Y_{k+1} \mid \mathcal{F}_{k-1}] + \Lambda^{3\top}_{k}\mathbb{E}[Y_{ k+1}\omega_{k} \mid \mathcal{F}_{k-1}] + \Lambda^{6}_{k}\boldsymbol{I}_{n}x_{k} \\&+ \Lambda^{7}_{k}\pi_{k}+\lambda^{2}_{k} \\
	=& \Lambda_k^{1\top} \mathbb{E}[T_{k+1} x_{k+1} + \phi_{k+1} \mid \mathcal{F}_{k-1}] + \Lambda_k^{3\top} \mathbb{E}[(T_{k+1} x_{k+1} \\&+ \phi_{k+1}) \omega_k \mid \mathcal{F}_{k-1}]  + \Lambda_k^6\boldsymbol{I}_{n} x_k + \Lambda_k^7 (\Pi_{k} x_k + \Sigma_{k}) + \lambda_k^2 \\
	=& \Lambda_k^{1\top} \mathbb{E}\bigg[T_{k+1} \biggl\{\Big(A_k + \Lambda_k^2 \Pi_k + D_k \omega_k + \Lambda_k^4 \Pi_k \omega_k\Big) x_k \\
	&+ \Lambda_k^2 \Sigma_k + b_k +  \Lambda_k^4 \Sigma_k \omega_k + \sigma_k \omega_k\biggr\} + \phi_{k+1} \mid \mathcal{F}_{k-1}\bigg] \\
	&+ \Lambda_k^{3\top} \mathbb{E}\bigg[(T_{k+1} \biggl\{\Big(A_k + \Lambda_k^2 \Pi_k + D_k \omega_k + \Lambda_k^4 \Pi_k \omega_k\Big)\\
	& \times x_k + \Lambda_k^2 \Sigma_k + b_k +  \Lambda_k^4 \Sigma_k \omega_k + \sigma_k \omega_k\biggr\} + \phi_{k+1}) \omega_k \mid \mathcal{F}_{k-1}\bigg] \\&+ \Lambda_k^6\boldsymbol{I}_{n} x_k + \Lambda_k^7 \big(\Pi_k x_k + \Sigma_k\big) + \lambda_k^2 \\
	=& \left[\Gamma(T_{k+1}) + \Upsilon(T_{k+1}) \Pi_{k+1} \right] x_k + \Phi(T_{k+1}, \phi_{k+1}) + \Upsilon(T_{k+1}) \Sigma_{k+1}.
\end{align*}
Due to the arbitrariness of \(x_k\), \eqref{eq:5.17} and \eqref{eq:5.19} hold.~Thus, both conditions (i)-(ii) are necessary.

\end{proof}

\section{Conclusion}
		
This paper investigates the open-loop Nash equilibrium for discrete-time linear quadratic (LQ) non-zero-sum difference games with random coefficients. The introduction of stochastic terms leads to results, particularly the form of the Riccati equations, that differ significantly from the deterministic case \cite{9}. The contributions of this paper are twofold. First, the stochastic Riccati equations are inherently more complex, as they include additional terms that account for the randomness of the coefficients. Specifically, unlike the deterministic case where only the conditional expectation term \(  \mathbb{E}[T_{k+1} \mid \mathcal{F}_{k-1}]  \) appears, the stochastic setting introduces some additional terms: \( \mathbb{E}[T_{k+1} \omega_k \mid \mathcal{F}_{k-1}] \) and \( \mathbb{E}[T_{k+1} \omega_k^2 \mid \mathcal{F}_{k-1}] \), which reflect the influence of the multiplicative noise and its interaction with the state and control weighting matrices. Second, the state feedback representation of the Nash equilibrium simplifies the complexity of the stochastic Hamiltonian systems for each player. We introduce the G\^{a}teaux derivative of each player's cost functional and use it to establish the necessary and sufficient conditions for the existence of a Nash equilibrium point. Also, we integrate two fully coupled forward and backward stochastic difference equations, collectively called the stochastic Hamiltonian system, into one and solve for the Nash equilibrium point using stochastic difference Riccati equations. However, due to the simultaneous appearance of the Nash equilibrium point in both the state and adjoint equations, decoupling the Hamiltonian systems leads to solutions represented by fully coupled non-symmetric Riccati equations, which poses significant challenges in addressing complex system problems. Ultimately, we derive an explicit expression for the Nash equilibrium point through state feedback representation. The insights from this research hold important implications for theoretical advancements and practical applications in control theory and game theory. There are many issues in discrete-time systems with random coefficients remain to be explored, including: (i) establishing verifiable existence conditions for solutions to the coupled Riccati equations, particularly focusing on the regularity requirements for random coefficient matrices; (ii) investigating discrete-time zero-sum Stackelberg games with random coefficients; (iii) extending the analysis to LQ non-zero-sum and zero-sum Nash games over infinite time horizons; and (iv) studying LQ Nash games for discrete-time Markov jump linear systems with random coefficients. These open problems will constitute our primary research agenda moving forward.

		%

\section{Declarations}
		
\textbf{Declaration of competing interest}:
		The authors declare that they have no known competing financial interests or personal relationships that could have appeared
		to influence the work reported in this paper.\\
		\textbf{Funding}: This research was supported by the Key Projects of Natural Science Foundation of Zhejiang Province (No. LZ22A010005), the National Natural Science Foundation of China (No.12271158), RGC of Hong Kong grants 15221621, 15226922 and 15225124, and PolyU 1-ZVXA. The funding played no role in the design, execution, analysis, or interpretation of the study. \\
		\textbf{Availability of Data and Materials}:
		We do not have any research data outside the submitted manuscript file. All of the material is owned by the authors and/or no permissions are required.

		\appendix
		\renewcommand{\theequation}{7.\arabic{equation}} 
		\setcounter{equation}{0} 
		\section*{Appendix \\\\\midsmall{A.~Compact Representation for FBS$\Delta$E of Two players}}
	\addcontentsline{toc}{section}{Appendix A: Compact Representation for FBS$\Delta$E  of Two players}
		\normalsize{The following matrix representation is introduced for more compact representation of FBS$\Delta$E :}
\small{
	\begin{equation}\label{eq:7.1}
		\left\{
		\begin{aligned}
			\Lambda^{1}_{k} &\equiv \begin{pmatrix} B_{k} & 0 \\ 0 & C_{k} \end{pmatrix}_{2n\times(m+l)}, \quad
			\Lambda^{2}_{k} \equiv \begin{pmatrix} B_{k} & C_{k} \end{pmatrix}_{n\times(m+l)}, \\[1ex]
			\Lambda^{3}_{k} &\equiv \begin{pmatrix} E_{k} & 0 \\ 0 & F_{k} \end{pmatrix}_{2n\times(m+l)},
			\quad\Lambda^{4}_{k} \equiv \begin{pmatrix} E_{k} & F_{k} \end{pmatrix}_{n\times(m+l)}, \\[1ex]
			\Lambda^{5}_{k} &\equiv \begin{pmatrix} Q_{k} \\ P_{k} \end{pmatrix}_{2n\times n}, \qquad\qquad\quad
			\Lambda^{6}_{k} \equiv \begin{pmatrix} L_{k} & 0 \\ 0 & M_{k} \end{pmatrix}_{(m+l)\times 2n},\\[1ex]
			\Lambda^{7}_{k} &\equiv \begin{pmatrix} R_{k} & 0 \\ 0 & S_{k} \end{pmatrix}_{(m+l)\times(m+l)}, \quad
			\pi_{k} \equiv \begin{pmatrix} u_{k} \\ v_{k} \end{pmatrix}_{(m+l)\times 1}, \\[1ex]
			Y_{k} &\equiv \begin{pmatrix} y^{1}_{k} \\ y^{2}_{k} \end{pmatrix}_{2n\times 1}, \qquad\qquad\qquad
			T_{k} \equiv \begin{pmatrix} T^{1}_{k} \\ T^{2}_{k} \end{pmatrix}_{2n\times n}, \\[1ex]
			\phi_{k} &\equiv \begin{pmatrix} \phi^{1}_{k} \\ \phi^{2}_{k} \end{pmatrix}_{2n\times 1},\qquad\qquad\qquad
			G \equiv \begin{pmatrix} G_{N} \\ H_{N} \end{pmatrix}_{2n\times n}, \\[1ex]
			g &\equiv \begin{pmatrix} g_{N} \\ h_{N} \end{pmatrix}_{2n\times 1}, \qquad\qquad\qquad
			\boldsymbol{I}_{n} \equiv \begin{pmatrix} I_{n} \\ I_{n} \end{pmatrix}_{2n\times n}, \\[1ex]
			\lambda^{1}_{k} &\equiv \begin{pmatrix} q_{k} \\ p_{k} \end{pmatrix}_{2n\times 1}, \qquad\qquad\qquad
			\lambda^{2}_{k} \equiv \begin{pmatrix} \rho_{k} \\ \theta_{k} \end{pmatrix}_{(m+l)\times 1},\\[1ex]
			\tilde{A}_{k} &\equiv \begin{pmatrix} A_{k} & 0 \\ 0 & A_{k} \end{pmatrix}_{2n\times 2n},
			\qquad\quad\tilde{D}_{k} \equiv \begin{pmatrix} D_{k} & 0 \\ 0 & D_{k} \end{pmatrix}_{2n\times 2n}, \\[1ex]
			f_{k} &\equiv \begin{pmatrix} f^{1}_{k} \\ f^{2}_{k} \end{pmatrix}_{2n\times 1}, \quad
			g_{k} \equiv \begin{pmatrix} g^{1}_{k} & 0 \\ 0 & g^{2}_{k} \end{pmatrix}_{2n\times 2n},
			h_{k} \equiv \begin{pmatrix} h^{1}_{k} & 0 \\ 0 & h^{2}_{k} \end{pmatrix}_{2n\times 2n}.
		\end{aligned}
		\right.
	\end{equation}
}
				\section*{\midsmall{B.~Compact Representation for Non-Symmetric Riccati Equation}} \normalsize{For $T_{k+1}\in\mathbb{R}^{2n},\phi_{k}\in\mathbb{R}^{2n}$, the following symbols are used to represent the non-symmetric Riccati equation:}
					\small{\begin{equation}\nonumber
						\left\{\begin{array}{lll}
							\begin{split}
								\Delta\left(T_{k+1}\right)&:=\Lambda^{5}_{k}+\tilde{A}_{k}^{\top}\mathbb{E}[T_{k+1} \mid \mathcal{F}_{k-1}]A_{k}\\&~~~~+\tilde{A}_{k}^{\top}\mathbb{E}[T_{k+1}\omega_{k} \mid \mathcal{F}_{k-1}]D_{k}\\&~~~~+\tilde{D}_{k}^{\top}\mathbb{E}[T_{k+1}\omega_{k} \mid \mathcal{F}_{k-1}]A_{k}\\&~~~~+\tilde{D}_{k}^{\top}\mathbb{E}[T_{k+1}\omega_{k}^{2} \mid \mathcal{F}_{k-1}]D_{k},
								\\\\
								\mathscr{L}\left(T_{k+1}\right) &:= \Lambda^{6}_{k}+\Lambda^{2\top}_{k}\mathbb{E}[T_{k+1} \mid \mathcal{F}_{k-1}]\tilde{A}_{k}\\&~~~~+\Lambda^{4\top}_{k}\mathbb{E}[T_{k+1}\omega_{k} \mid \mathcal{F}_{k-1}]\tilde{A}_{k}\\&~~~~+\Lambda^{2\top}_{k}\mathbb{E}[T_{k+1}\omega_{k} \mid \mathcal{F}_{k-1}]\tilde{D}_{k}\\&~~~~+\Lambda^{4\top}_{k}\mathbb{E}[T_{k+1}\omega_{k}^{2} \mid \mathcal{F}_{k-1}]\tilde{D}_{k}
								,
							\end{split}
						\end{array}
						\right.
				\end{equation}}
		\small{\begin{equation}\label{eq:7.2}
			\left\{\begin{array}{lll}
				\begin{split}
					\Upsilon\left(T_{k+1}\right)&:=\Lambda^{7}_{k}+\Lambda^{1\top}_{k}\mathbb{E}[T_{k+1} \mid \mathcal{F}_{k-1}]\Lambda^{2}_{k}\\&~~~~+\Lambda^{1\top}_{k}\mathbb{E}[T_{k+1}\omega_{k} \mid \mathcal{F}_{k-1}]\Lambda^{4}_{k}\\&~~~~+\Lambda^{3\top}_{k}\mathbb{E}[T_{k+1}\omega_{k} \mid \mathcal{F}_{k-1}]\Lambda^{2}_{k}\\&~~~~+\Lambda^{3\top}_{k}\mathbb{E}[T_{k+1}\omega^{2}_{k} \mid \mathcal{F}_{k-1}]\Lambda^{4}_{k},
					\\\\
					\Gamma\left(T_{k+1}\right)&:=\Lambda^{6}_{k}\boldsymbol{I}_{n}+\Lambda^{1\top}_{k}\mathbb{E}[T_{k+1} \mid \mathcal{F}_{k-1}]A_{k}\\&~~~~+\Lambda^{3\top}_{k}\mathbb{E}[T_{k+1}\omega_{k} \mid \mathcal{F}_{k-1}]A_{k}\\&~~~~+\Lambda^{1\top}_{k}\mathbb{E}[T_{k+1}\omega_{k} \mid \mathcal{F}_{k-1}]D_{k}\\&~~~~+\Lambda^{3\top}_{k}\mathbb{E}[T_{k+1}\omega^{2}_{k} \mid \mathcal{F}_{k-1}]D_{k},
					\\\\
					\Theta\left(T_{k+1},\phi_{k+1}\right) &:= \lambda^{1}_{k}+\tilde{A}^{\top}_{k}\big(\mathbb{E}[T_{k+1}\mid \mathcal{F}_{k-1}]b_{k}\\&~~~~+\mathbb{E}[T_{k+1}\omega_{k}\mid \mathcal{F}_{k-1}]\sigma_{k}+\mathbb{E}[\phi_{k+1}\mid \mathcal{F}_{k-1}]\big)\\&~~~~+\tilde{D}^{\top}_{k}\big(\mathbb{E}[T_{k+1}\omega_{k}\mid \mathcal{F}_{k-1}]b_{k}\\&~~~~+\mathbb{E}[T_{k+1}\omega_{k}^{2}\mid \mathcal{F}_{k-1}]\sigma_{k}+\mathbb{E}[\phi_{k+1}\omega_{k}\mid \mathcal{F}_{k-1}]\big),
					\\\\
					\Phi\left(T_{k+1},\phi_{k+1}\right)&:=\lambda^{2}_{k}+\Lambda^{1\top}_{k}\big(\mathbb{E}[T_{k+1}\mid \mathcal{F}_{k-1}]b_{k}\\&~~~~+\mathbb{E}[T_{k+1}\omega_{k}\mid \mathcal{F}_{k-1}]\sigma_{k}+\mathbb{E}[\phi_{k+1}\mid \mathcal{F}_{k-1}]\big)\\&~~~~+\Lambda^{3\top}_{k}\big(\mathbb{E}[T_{k+1}\omega_{k}\mid \mathcal{F}_{k-1}]b_{k}\\&~~~~+\mathbb{E}[T_{k+1}\omega_{k}^{2}\mid \mathcal{F}_{k-1}]\sigma_{k}+\mathbb{E}[\phi_{k+1}\omega_{k}\mid \mathcal{F}_{k-1}]\big).
				\end{split}
			\end{array}
			\right.
		\end{equation}}

\end{document}